\RequirePackage{fix-cm}
\documentclass[smallextended]{svjour3}       
\usepackage[T1]{fontenc}
\usepackage[utf8]{inputenc}

\smartqed  
\usepackage{graphicx}
\usepackage{multirow}
\usepackage{multicol}
\usepackage{amsmath}
\usepackage{color}
\usepackage{hyperref}
\usepackage{newtxtext,newtxmath}

\usepackage{amsfonts}

\newtheorem{assumption}{Assumption}

\DeclareMathOperator*{\argmin}{argmin}
\begin{document}

\title{Moreau Envelope Augmented Lagrangian Method for Nonconvex Optimization with Linear Constraints
\thanks{We thank Kaizhao Sun for discussions that help us complete this paper, as well as presenting to us an additional approach to ensure boundedness. The work of J. Zeng is partly supported by National Natural Science Foundation of China (No. 61977038) and the Thousand Talents Plan of Jiangxi Province
(No. jxsq2019201124). The work of D.-X. Zhou is partly supported by Research Grants Council of Hong Kong (No. CityU 11307319), Laboratory for AI-powered Financial Technologies, and the Hong Kong Institute
for Data Science.}
}


\author{Jinshan Zeng \and
        Wotao Yin \and
        Ding-Xuan Zhou 
}


\institute{J. Zeng \at
              School of Computer and Information Engineering, Jiangxi Normal University, Nanchang, China.\\
              Liu Bie Ju Centre for Mathematical Sciences, City University of Hong Kong, Hong Kong.\\
              \email{jinshanzeng@jxnu.edu.cn}           
           \and
           W. Yin \at
           Department of Mathematics, University of California, Los Angeles, CA. \\
           \email{wotaoyin@math.ucla.edu}
           \and
           D.X. Zhou \at
           School of Data Science, Department of Mathematics, and Liu Bie Ju Centre for Mathematical Sciences, City University of Hong Kong, Hong Kong. \\
           \email{mazhou@cityu.edu.hk}
}

\date{Received: date / Accepted: date}

\maketitle

\begin{abstract}
The augmented Lagrangian method (ALM) is one of the most useful methods for constrained optimization. Its convergence has been well established under convexity assumptions or smoothness assumptions, or under both assumptions. ALM may experience oscillations and divergence when the underlying problem is simultaneously nonconvex and nonsmooth. In this paper, we consider the linearly constrained problem with a nonconvex (in particular, weakly convex) and nonsmooth objective. We modify ALM to use a Moreau envelope of the augmented Lagrangian and establish its convergence under conditions that are weaker than those in the literature. We call it the \textit{Moreau envelope augmented Lagrangian (MEAL)} method. We also show that the iteration complexity of MEAL is $o(\varepsilon^{-2})$ to yield an $\varepsilon$-accurate first-order stationary point. We establish its whole sequence convergence (regardless of the initial guess) and a rate when a Kurdyka-{\L}ojasiewicz property is assumed. Moreover, when the subproblem of MEAL has no closed-form solution and is difficult to solve, we propose two practical variants of MEAL, an inexact version called \textit{iMEAL} with an approximate proximal update, and a linearized version called \textit{LiMEAL} for the constrained problem with a composite objective. Their convergence is also established.

\keywords{Nonconvex nonsmooth optimization \and augmented Lagrangian method \and Moreau envelope \and proximal augmented Lagrangian method \and Kurdyka-{\L}ojasiewicz inequality
}
\end{abstract}


\section{Introduction}

In this paper, we consider the following optimization problem with linear constraints
\begin{equation}
\label{Eq:problem}
\begin{array}{ll}
\mathrm{minimize}_{x\in \mathbb{R}^n} & f(x) \\
\mathrm{subject \ to} & Ax=b,
\end{array}
\end{equation}
where
$f: \mathbb{R}^n \rightarrow \mathbb{R}$ is a proper, lower-semicontinuous \textit{weakly convex} function, which is possibly nonconvex
and nonsmooth,
$A\in \mathbb{R}^{m\times n}$ and $b\in \mathbb{R}^m$ are some given matrix and vector, respectively.
A function $f$ is said to be \textit{weakly convex} with a modulus $\rho>0$ if $f(x)+\frac{\rho}{2}\|x\|^2$ is convex on $\mathbb{R}^n$,
where $\|\cdot\|$ is the Euclidean norm.
The class of weakly convex functions is broad \cite{Nurminskii73}, including all convex functions, smooth but nonconvex functions with Lipschitz continuous gradient, and their composite forms (say, $f(x)=h(x)+g(x)$ with both $h$ and $g$ being weakly convex, and $f(x)=g(h(x))$ with $g$ being convex and Lipschitz continuous and $h$ being a smooth mapping with Lipschitz Jacobian \cite[Lemma 4.2]{Drusvyatskiy-Paquette19}).

The augmented Lagrangian method (ALM) is a well-known algorithm for constrained optimization by Hestenes \cite{Hestenes69} and Powell \cite{Powell69}.
ALM has been extensively studied and has a large body of literature (\cite{Bertsekas73,Birgin10,Conn91,Conn96,Rockafellar73-ALM} just to name a few),
yet \emph{no ALM algorithm can solve the underlying problem (\ref{Eq:problem}) without at least one of the following assumptions}: convexity \cite{Bertsekas73,Bertsekas76,Fernadez12,Polyak-Tretyakov73,Rockafellar73-ALM}, or smoothness \cite{Andreani08,Andreani10,Andreani19,Andreani18,Curtis15}, or solving nonconvex subproblems to their global minima \cite{Birgin10,Birgin18}, or an auto-updated penalty sequence staying bounded on the problem at hand \cite{Birgin20,Grapiglia-Yuan19}.
Indeed, without these assumptions, ALM may oscillate and even diverge unboundedly on simple quadratic programs~\cite{Wang19,Zhang-Luo18} on weakly convex objectives. An example is given Sec. \ref{sc:exp1} below.

At a high level, we introduce a Moreau-envelope modification of the ALM for solving \eqref{Eq:problem} and show the method can converge under weaker conditions. 
In particular, convexity is relaxed to weak convexity; nonsmooth functions are allowed; the subproblems can be solved inexactly to some extent; linearization can be applied to the Lipschitz-differential function in the objective; and, there is no assumption on the rank of $A$. On the other hand, we introduce two alternative subgradient properties in Definition \ref{Def:implicit-Lip-bounded-subgrad} below
as our main assumption. By also assuming either a bounded energy sequence or bounded primal-dual sequence, we derive certain subsequence rates of convergence.
We introduce a novel way to establish those boundedness properties based on a feasible coercivity assumption and a local-stability assumption on the subproblem. Finally, with the additional assumption of Kurdyka-{\L}ojasiewicz (K{\L}) inequality, we establish global convergence.
Overall, this paper shows that the Moreau envelope technique makes ALM applicable to more problems.
\subsection{Proposed Algorithms}
\label{sc:algorithms}
To present our algorithm, define the augmented Lagrangian:
\begin{equation}
\label{Eq:augmented-Lagrangian}
{\cal L}_{\beta} (x,\lambda) := f(x)+\langle \lambda, Ax-b \rangle + \frac{\beta}{2}\|Ax-b\|^2,
\end{equation}
and the \emph{Moreau envelope}
of ${\cal L}_{\beta}(x,\lambda)$:
\begin{equation}
\label{Eq:Moreauenvelope_AL}
\phi_{\beta}(z,\lambda) = \min_{x} \left\{{\cal L}_{\beta}(x,\lambda)+\frac{1}{2\gamma}\|x-z\|^2\right\},
\end{equation}
where $\lambda \in \mathbb{R}^m$ is a multiplier vector, $\beta>0$ is a penalty parameter, and $\gamma>0$ is a proximal parameter. The Moreau envelope applies to the primal variable $x$ for each fixed dual variable $\lambda$.

We introduce \textit{Moreau Envelope Augmented Lagrangian method} (dubbed \textit{MEAL}) as follows: given an initialization $(z^0,\lambda^0)$, $\gamma>0$, a sequence of penalty parameters $\{\beta_k\}$ and a step size $\eta \in (0,2)$, for $k=0,1,\ldots,$ run
\begin{equation}
\label{alg:MEAL}
\mathrm{(MEAL)} \quad
\left\{
\begin{array}{l}
 z^{k+1} = z^k - \eta \gamma \nabla_z \phi_{\beta_k}(z^k,\lambda^k),\\
 \lambda^{k+1} = \lambda^k + \beta_k \nabla_\lambda \phi_{\beta_k}(z^k,\lambda^k).
\end{array}
\right.
\end{equation}
The penalty parameter $\beta_k$ can either vary or be fixed.

Introduce
\begin{equation*}
x^{k+1}= \mathrm{Prox}_{\gamma,{\cal L}_{\beta_k}(\cdot,\lambda^{k})}(z^k) := \argmin_x \left\{ {\cal L}_{\beta_k}(x,\lambda^k)+\frac{1}{2\gamma}\|x-z^k\|^2\right\}, \ \forall k\in \mathbb{N},
\end{equation*}
which yields $\nabla_z \phi_{\beta_k}(z^k,\lambda^k) = \gamma^{-1} (z^k-x^{k+1})$ and $\nabla_\lambda \phi_{\beta_k}(z^k,\lambda^k) = A x^{k+1} - b$.
Then, MEAL \eqref{alg:MEAL} is equivalent to:
\begin{equation}
\label{alg:MEAL-reformulation}
\mathrm{(MEAL\ Reformulated)} \quad
\left\{
\begin{array}{l}
x^{k+1} = \mathrm{Prox}_{\gamma,{\cal L}_{\beta_k}(\cdot,\lambda^{k})}(z^k),\\
z^{k+1} = z^k -\eta (z^k - x^{k+1}),\\
\lambda^{k+1} = \lambda^k + \beta_k (Ax^{k+1}-b).
\end{array}
\right.
\end{equation}
Next, we provide two practical variants of MEAL that do not require an accurate computation of $\mathrm{Prox}_{\gamma,{\cal L}_{\beta}}$.

\paragraph{Inexact MEAL (iMEAL)}
We call $x^{k+1}$ an $\epsilon_k$-accurate stationary point of the $x$-subproblem in \eqref{alg:MEAL-reformulation} if there exists
\begin{equation}\label{iMealCond}
s^k \in \partial_x {\cal L}_{\beta_k}(x^{k+1},\lambda^k) + \gamma^{-1}(x^{k+1}-z^k)\quad\text{such that}~
\|s^k\| \leq \epsilon_k.
\end{equation}
\textit{iMEAL} is described as follows: given an initialization $(z^0,\lambda^0)$, $\gamma>0$, $\eta \in (0, 2)$, and  two positive sequences $\{\epsilon_k\}$ and $\{\beta_k\}$, for $k=0,1,\ldots,$ run
\begin{equation}
\label{alg:iMEAL}
\mathrm{(iMEAL)} \quad
\left\{
\begin{array}{l}
\mathrm{find \ an} \ x^{k+1} \ \mathrm{to\ satisfy} \ \eqref{iMealCond},\\
z^{k+1} = z^k -\eta (z^k - x^{k+1}),\\
\lambda^{k+1} = \lambda^k + \beta_k (Ax^{k+1}-b).
\end{array}
\right.
\end{equation}

\paragraph{Linearized MEAL (LiMEAL)}
When problem \eqref{Eq:problem} has the following form
\begin{equation}
\label{Eq:problem-CP}
\begin{array}{ll}
\mathop{\mathrm{minimize}}_{x\in \mathbb{R}^n} & f(x):= h(x) + g(x)\\
\mathrm{subject \ to} & Ax=b,
\end{array}
\end{equation}
where $h:\mathbb{R}^n \rightarrow \mathbb{R}$ is Lipschitz-continuous differentiable and $g:\mathbb{R}^{n} \rightarrow \mathbb{R}$ is weakly convex and has an easy proximal operator (in particular, admitting a closed-form solution) \cite{Hajinezhad-Hong19,Wang19,Xu-Yin-BCD13,Zeng-DGD18},
we shall use $\nabla h$.
Write $f^k(x):= h(x^k)+\langle \nabla h(x^k), x - x^k\rangle + g(x)$
and
$
{\cal L}_{\beta,{f^k}}(x,\lambda):= f^k(x)+\langle \lambda, Ax-b \rangle + \frac{\beta}{2}\|Ax-b\|^2.
$
We describe
\textit{LiMEAL} for \eqref{Eq:problem-CP} as: given $(z^0,\lambda^0)$, $\gamma>0$, $\eta \in (0,2)$ and $\{\beta_k\}$, for $k=0,1,\ldots,$ run
\begin{equation}
\label{alg:LiMEAL}
\mathrm{(LiMEAL)} \quad
\left\{
\begin{array}{l}
x^{k+1} = \mathrm{Prox}_{\gamma,{\cal L}_{\beta_k,{f^k}}(\cdot,\lambda^k)}(z^k),\\
z^{k+1} = z^k - \eta(z^k-x^{k+1}),\\
\lambda^{k+1} = \lambda^k + \beta_k (Ax^{k+1}-b).
\end{array}
\right.
\end{equation}
Since one can choose to use $h$ or not in LiMEAL, LiMEAL is more general than MEAL.

\subsection{Relation to ALM and Proximal ALM}
\label{sc:relation-existing-methods}

Like ALM, MEAL alternatively updates primal and dual variables; but unlike ALM, MEAL applies the update to the Moreau envelope of augmented Lagrangian.
By \cite{Rockafellar-var97}, the Moreau envelope $\phi_{\beta_k}(z,\lambda^k)$ provides a smooth approximation of ${\cal L}_{\beta_k}(x,\lambda^k)$ from below and shares the same minima.
The smoothness of Moreau envelope alleviates the possible oscillation that arises when ALM is applied to certain nonconvex optimization problems. 

For the problems satisfying the conditions in this paper, ALM may require a sequence of possibly unbounded $\{\beta_k\}$. When $\beta_k$ is large, the ALM subproblem is ill-conditioned.
Therefore, bounding $\beta_k$ is practically desirable \cite{Birgin-book14,Conn91}.
MEAL and its practical variants can use a fixed penalty parameter under a novel subgradient assumption in Definition \ref{Def:implicit-Lip-bounded-subgrad} later.

Proximal ALM was introduced in \cite{Rockafellar76-PALM}. Its variants were recently studied in \cite{Hajinezhad-Hong19,Hong17-Prox-PDA,Zhang-Luo20,Zhang-Luo18}.
These methods
add a proximal term to the augmented Lagrangian.
Under the reformulation \eqref{alg:MEAL-reformulation},
proximal ALM \cite{Rockafellar76-PALM} for problem \eqref{Eq:problem} is a special case of MEAL with the step size $\eta =1$. 
In \cite{Hong17-Prox-PDA}, a proximal primal-dual algorithm called \textit{Prox-PDA} was proposed for problem \eqref{Eq:problem}.
Certain non-Euclidean matrix norms were adopted in Prox-PDA to guarantee the strong convexity of the ALM subproblem.
A proximal linearized version of Prox-PDA for the composite optimization problem \eqref{Eq:problem-CP} was studied in \cite{Hajinezhad-Hong19}. These methods are closely related to MEAL, but their convergence conditions in the literature are stronger.

Recently, \cite{Zhang-Luo20,Zhang-Luo18} modified proximal inexact ALM for the linearly constrained problems with an additional bounded box constraint set or polyhedral constraint set, denoted by ${\cal C}$. Our method is partially motivated by their methods.
Their problems are equivalent to the composite optimization problems \eqref{Eq:problem-CP} with $g(x) = \iota_{\cal C}(x)$, where $\iota_{\cal C}(x)=0$ when $x\in {\cal C}$ and $+\infty$ otherwise.
In this setting, the methods in \cite{Zhang-Luo20,Zhang-Luo18} can be regarded as prox-linear versions of LiMEAL \eqref{alg:LiMEAL}, that is, yielding $x^{k+1}$ via a prox-linear scheme \cite{Xu-Yin-BCD13} instead of the minimization scheme as used in LiMEAL \eqref{alg:LiMEAL}, together with an additional dual step size and a sufficiently small primal step size in \cite{Zhang-Luo20,Zhang-Luo18}. Specifically, in the case of $g(x) = \iota_{\cal C}(x)$, the updates of $x^{k+1}$ in methods in \cite{Zhang-Luo20,Zhang-Luo18} are yielded by
\begin{align*}
    x^{k+1} = \mathrm{Proj}_{\cal C}(x^k - s\nabla K(x^k,z^k,\lambda^k)),
\end{align*}
where $K(x^k,z^k,\lambda^k)) = {\cal L}_{\beta^k,f}(x,\lambda^k) + \frac{1}{2\gamma}\|x-z^k\|^2$, and $\mathrm{Proj}_{\cal C}(x)$ is the projection of $x$ onto ${\cal C}$.
Besides the difference, LiMEAL can handle proximal functions beyond the indicator function and permits a wider choice $\eta \in (0,2)$.

\subsection{Other Related Literature}

On convex and constrained problems,
locally linear convergence\footnote{Locally linear convergence means exponentially fast convergence to a local minimum from a sufficiently close initial point.} of ALM has been extensively studied in the literature \cite{Bertsekas73,Bertsekas76,Bertsekas82,Conn00,Fernadez12,Nocedal99,Polyak-Tretyakov73}, mainly under the second order sufficient condition (SOSC) and constraint conditions such as the linear independence constraint qualification (LICQ).
Global convergence (i.e., convergence regardless of the initial guess) of ALM and its variants were studied in~\cite{Andreani07,Armand17,Birgin05,Birgin12,Birgin10,Conn91,Conn96,Rockafellar73-ALM,Tretykov73}, mainly under constraint qualifications and assumed boundedness of nondecreasing penalty parameters.
On nonconvex and constrained problems, convergence of ALM was recently studied in \cite{Andreani08,Andreani10,Andreani19,Andreani18,Birgin10,Birgin18,Curtis15}, mainly under the following assumptions: solving nonconvex subproblems to their approximate global minima or stationary points \cite{Birgin10,Birgin18}, or
boundedness of the nondecreasing penalty sequence \cite{Birgin20,Grapiglia-Yuan19}.
Most of them require \textit{Lipschitz differentiability} of the objective.

Convergence of proximal ALM and its variants was established under the assumptions of either convexity in \cite{Rockafellar76-PALM} or smoothness (in particular, Lipschitz differentiablity) in \cite{Hajinezhad-Hong19,Hong17-Prox-PDA,Jiang19,Xie-Wright19,Zhang-Luo20,Zhang-Luo18}.
Besides proximal ALM, other related works for nonconvex and constrained problems include \cite{Bian15,Haeser19,Nouiehed18,ONeill20}, which also assume smoothness of the objective, plus either gradient or Hessian information.

\subsection{Contribution and Novelty}
MEAL, iMEAL and LiMEAL achieve the same order of iteration complexity $o({\varepsilon^{-2}})$ to reach an $\varepsilon$-accurate first-order stationary point, slightly better than those in the ALM literature~\cite{Hajinezhad-Hong19,Hong17-Prox-PDA,Xie-Wright19,Zhang-Luo18,Zhang-Luo20} while also requiring weaker conditions. Our methods have convergence guarantees for a broader class of objective functions, for example, nonsmooth and nonconvex functions like 
the smoothly clipped absolute deviation (SCAD) regularization \cite{Fan-SCAD} and minimax concave penalty (MCP) regularization \cite{Zhang-MCP},  which are underlying the applications of 
statistical learning and beyond \cite{Wang19}. 

Note that we only assume the feasibility of $Ax=b$, which is weaker than the commonly-used hypotheses such as: the strict complementarity condition in \cite{Zhang-Luo18},  certain rank assumption (such as $\mathrm{Im}(A)\subseteq \mathrm{Im}(B)$ when considering the two- (multi-) block case $Ax+By=0$) in \cite{Wang19}, and the linear independence constrained qualification (LICQ) in \cite{Bertsekas82,Nocedal99} (which implies the full-rank assumption in the linear constraint case).

Our analysis is noticeably different from those in the literature~\cite{Rockafellar76-PALM,Hajinezhad-Hong19,Hong17-Prox-PDA,Jiang19,Zhang-Luo18,Zhang-Luo20,Xie-Wright19,Wang19}. We base our analysis on new potential functions. The Moreau envelope in the potential functions is partially motivated by~\cite{Davis-Drusvyatskiy19}. Our overall potential functions are new and tailored for MEAL, iMEAL, and LiMEAL and include the augmented Lagrangian with additional terms. The technique of analysis may have its own value for further generalizing and improving ALM-type methods.

\subsection{Notation and Organization}

We let $\mathbb{R}$ and $\mathbb{N}$ denote the sets of real and natural numbers, respectively. Given a matrix $A$, $\mathrm{Im}(A)$ denotes its image, and $\tilde{\sigma}_{\min}(A^TA)$ denotes the smallest positive eigenvalue of $A^TA$. $\|\cdot\|$ is the Euclidean norm for a vector.
Given any two nonnegative sequences $\{\xi_k\}$ and $\{\zeta_k\}$, we write $\xi_k = o(\zeta_k)$ if $\lim_{k\rightarrow \infty} \frac{\xi_k}{\zeta_k}=0$, and $\xi_k = {\cal O}(\zeta_k)$ if there exists a positive constant $c$ such that $\xi_k \leq c \zeta_k$ for all sufficiently large $k$.

In the rest of this paper,
Section \ref{sc:preliminary} presents background and preliminary techniques.
Section \ref{sc:convergence-MEAL} states convergence results of MEAL and iMEAL.
Section \ref{sc:LiMEAL} presents the results of LiMEAL.
Section \ref{sc:main-proofs} includes main proofs.
Section \ref{sc:discussion} provides sufficient conditions for certain boundedness assumptions in above results along with comparisons with the related work. Section \ref{sc:experiment} provides some numerical experiments to demonstrate the effectiveness of proposed methods.
We conclude this paper in Section \ref{sc:conclusion}.

\section{Background and Preliminaries}
\label{sc:preliminary}

This paper uses extended-real-valued functions, for example, $h:\mathbb{R}^n \to \mathbb{R} \cup \{+\infty\}$.
Write the domain of $h$ as $\mathrm{dom}(h):=\{x\in \mathbb{R}^n: h(x)<+\infty\}$ and its range as $\mathrm{ran}(h):= \{y: y=h(x), \forall x\in \mathrm{dom}(h)\}$.
For each $x\in \mathrm{dom}(h)$, the \textit{Fr\'{e}chet subdifferential} of $h$ at $x$, written as $\widehat{\partial}h(x)$, is the set of vectors $v\in \mathbb{R}^n$ satisfying
\[
\liminf_{u\neq x, u\rightarrow x} \ \frac{h(u)-h(x)-\langle v,u-x\rangle}{\|x-u\|} \geq 0.
\]
When $x\notin \mathrm{dom}(h),$ we define
$\widehat{\partial} h(x) = \emptyset.$
The \emph{limiting-subdifferential} (or simply \emph{subdifferential}) of $h$~\cite{Mordukhovich-2006} at $x\in \mathrm{dom}(h)$ is defined
as
\begin{equation}
\label{Def:limiting-subdifferential}
\partial h(x) := \{v\in \mathbb{R}^n: \exists x^t \to x,\; h(x^t)\to h(x), \;  \widehat{\partial} h(x^t) \ni v^t \to v\}.
\end{equation}
A necessary (but not sufficient) condition for $x\in \mathbb{R}^n$ to be a minimizer of $h$ is $0 \in \partial h(x)$.
A point that satisfies this inclusion is called \textit{limiting-critical} or simply \textit{critical}.
The distance between a point $x$ and a subset ${\cal S}$ of $\mathbb{R}^n$ is defined
as $\mathrm{dist}(x, {\cal S}) = \inf_u \{\|x-u\|: u\in {\cal S}\}$.

\subsection{Moreau Envelope}
\label{sc:moreau-envelope}

Given a function $h: \mathbb{R}^n \rightarrow \mathbb{R}$, define its \textit{Moreau envelope} \cite{Moreau65,Rockafellar-var97}:
\begin{equation}
\label{Eq:Moreau-envelope}
{\cal M}_{\gamma,h}(z) = \min_{x} \left\{ h(x) + \frac{1}{2\gamma}\|x-z\|^2\right\},
\end{equation}
where $\gamma>0$ is a parameter. Define its associated proximity operator
\begin{equation}
\label{Eq:prox-operator}
\mathrm{Prox}_{\gamma,h}(z) = \argmin_{x} \left\{h(x) + \frac{1}{2\gamma}\|x-z\|^2\right\}.
\end{equation}
If $h$ is $\rho$-weakly convex and $\gamma\in (0,\rho^{-1})$, then $\mathrm{Prox}_{\gamma,h}$ is monotone, single-valued, and Lipschitz, and ${\cal M}_{\gamma,h}$ is differentiable with
\begin{equation}
\label{Eq:Moreau-gradient}
\nabla {\cal M}_{\gamma,h}(z) = \gamma^{-1}\left(z-\mathrm{Prox}_{\gamma,h}(z)\right)\in \partial h(\mathrm{Prox}_{\gamma,h}(z));
\end{equation}
see \cite[Proposition 13.37]{Rockafellar-var97}.
From~\cite{Drusvyatskiy18,Drusvyatskiy-Paquette19}, we also have
\begin{align*}
&{\cal M}_{\gamma,h}(\mathrm{Prox}_{\gamma,h}(z)) \leq h(z), \nonumber\\
&\|\mathrm{Prox}_{\gamma,h}(z)-z\| = \gamma \|\nabla {\cal M}_{\gamma,h}(z)\|, \nonumber\\
&\mathrm{dist}(0,\partial h(\mathrm{Prox}_{\gamma,h}(z))) \leq \|\nabla {\cal M}_{\gamma,h}(z)\|. 
\end{align*}
The first relation above presents Moreau envelope as a smooth lower approximation of $h$. By the second and third relations, small $\|\nabla {\cal M}_{\gamma,h}(z)\|$ implies that $z$ is \textit{near} its proximal point $\mathrm{Prox}_{\gamma,h}(z)$ and $z$ is \textit{nearly stationary} for $h$ \cite{Davis-Drusvyatskiy19}.
Therefore, $\|\nabla {\cal M}_{\gamma,h}(z)\|$ can be used as a \textit{continuous stationarity measure}.
Hence, replacing the augmented Lagrangian with its Moreau envelope
not only generates a strongly convex subproblem
but also yields a stationarity measure.

\subsection{Implicit Regularity Properties}
\label{sc:implicit-property}

Let $h$ be a proper, lower semicontinuous, $\rho$-weakly convex function.
Given a $\gamma \in (0,\rho^{-1})$, define the \textit{generalized inverse mapping} of $\mathrm{Prox}_{\gamma,h}$:
\begin{align}
\label{Eq:prox-operator-inverse}
\mathrm{Prox}_{\gamma,h}^{-1}(x):=\{w:\mathrm{Prox}_{\gamma,h}(w) = x\}, \quad \forall x\in \mathrm{ran}(\mathrm{Prox}_{\gamma,h}).
\end{align}
In the definition below, we introduce two important regularity properties.
\begin{definition}
\label{Def:implicit-Lip-bounded-subgrad}
Let $h$ be a proper, lower semicontinuous and $\rho$-weakly convex function.
\begin{enumerate}
\item[(a)] We say $h$ satisfies the \textbf{implicit Lipschitz subgradient} property if for any $\gamma \in (0,\rho^{-1})$, there exists $L>0$ (depending on $\gamma$) such that for any $u,v\in \mathrm{ran}(\mathrm{Prox}_{\gamma,h})$,
    \[\|\nabla {\cal M}_{\gamma,h}(w) - \nabla {\cal M}_{\gamma,h}(w') \| \leq L\|u-v\|, \ \forall w\in \mathrm{Prox}_{\gamma,h}^{-1}(u), w'\in \mathrm{Prox}_{\gamma,h}^{-1}(v);\]

\item[(b)] We say $h$ satisfies the \textbf{implicit bounded subgradient} property if for any $\gamma \in (0,\rho^{-1})$, there exists $\hat{L}>0$ (depending on $\gamma$) such that for any $u\in \mathrm{ran}(\mathrm{Prox}_{\gamma,h})$,
    \[\|\nabla {\cal M}_{\gamma,h}(w)\| \leq \hat{L}, \ \forall w\in \mathrm{Prox}_{\gamma,h}^{-1}(u).\]
\end{enumerate}
\end{definition}
Since $\nabla {\cal M}_{\gamma,h}(x) \in \partial h(\mathrm{Prox}_{\gamma,h}(x))$ for any $x\in \mathbb{R}^n$,
we have $\nabla {\cal M}_{\gamma,h}(w) \in \partial h(u), \forall u \in \mathrm{ran}(\mathrm{Prox}_{\gamma,h})$ and $w\in \mathrm{Prox}_{\gamma,h}^{-1}(u)$.
Hence, the \textit{implicit Lipschitz subgradient} and \textit{implicit bounded subgradient} imply, respectively, the \textit{Lipschitz continuity} and \textit{boundedness} only on the components of $\partial h$ that are Moreau envelope gradients, but not on other components of $\partial h$.
When $h$ is differentiable, \textit{implicit Lipschitz subgradient} implies \textit{Lipschitz gradient}.
Having \textit{implicit bounded subgradients} is weaker than having bounded $\partial h$,
which is commonly assumed in the analysis of nonconvex algorithms (cf. \cite{Davis-Drusvyatskiy19,Hajinezhad-Hong19,Zeng-DGD18}). Nonsmooth and nonconvex functions like 
the SCAD regularization and MCP regularization which appear in  statistical learning \cite{Wang19}, have \textit{implicit bounded subgradients}. 

\subsection{Kurdyka-{\L}ojasiewicz Inequality}
\label{sc:KL-ineq}

The Kurdyka-{\L}ojasiewicz (K{\L}) inequality \cite{Bolte-KL2007a,Bolte-KL2007b,Kurdyka-KL1998,Lojasiewicz-KL1963,Lojasiewicz-KL1993}
is a property that leads to global convergence of nonconvex algorithms in the literature (see, \cite{Attouch13,Bolte2014,Wang19,Xu-Yin-BCD13,Zeng-BCD19,Zeng-ADMM19}).
The following definition of Kurdyka-{\L}ojasiewicz property is adopted from \cite{Bolte-KL2007a}.

\begin{definition}
\label{Def-KLProp}
A function	$h:\mathbb{R}^n \rightarrow  \mathbb{R}\cup \{+\infty\}$ is said to have the {Kurdyka-{\L}ojasiewicz  property} at $x^*\in \mathrm{dom}(\partial h)$ if there exist a neighborhood ${\cal U}$ of $x^*$, a constant $\nu>0$, and a continuous concave function $\varphi(s) = cs^{1-\theta}$ for some $c>0$ and $\theta \in [0,1)$ such that the Kurdyka-{\L}ojasiewicz inequality holds: for all $x \in {\cal U} \cap \mathrm{dom}(\partial h)$ and $h(x^*) < h(x) < h(x^*)+\nu$,
\begin{equation}
\varphi'(h(x)-h(x^*)) \cdot\mathrm{dist}(0,\partial h(x))\geq 1, 	\label{Eq:KLIneq}
\end{equation}
(we use the conventions: $0^0=1, \infty/\infty=0/0=0$),
where $\theta$ is called the K{\L} exponent of $h$ at $x^*$. Proper lower semicontinuous functions satisfying the K{\L} inequality at every point of $\mathrm{dom}(\partial h)$ are called K{\L} functions.
\end{definition}
This property was firstly introduced by \cite{Lojasiewicz-KL1993} on real analytic functions \cite{Krantz2002-real-analytic} for $\theta \in \left[ \tfrac{1}{2},1\right) $, was then extended to functions defined on the o-minimal structure in \cite{Kurdyka-KL1998}, and was later extended to nonsmooth subanalytic functions in \cite{Bolte-KL2007a}.
K{\L} functions include real analytic functions \cite{Krantz2002-real-analytic},
semialgebraic functions \cite{Bochnak-semialgebraic1998},
tame functions defined in some o-minimal structures \cite{Kurdyka-KL1998}, continuous subanalytic functions \cite{Bolte-KL2007a}, definable functions \cite{Bolte-KL2007b}, locally strongly convex functions \cite{Xu-Yin-BCD13}, as well as many
deep-learning training models~\cite{Zeng-BCD19,Zeng-ADMM19}.

\section{Convergence of MEAL}
\label{sc:convergence-MEAL}

This section presents the convergence results of MEAL and iMEAL. We postpone their proofs to Section \ref{sc:main-proofs}.

\subsection{Assumptions and Stationarity Measure}
\label{sc:MEAL-assump}

\begin{assumption}
\label{Assump:feasibleset}
The set ${\cal X}:=\{x:Ax=b\}$ is nonempty.
\end{assumption}

\begin{assumption}
\label{Assump:MEAL}
The objective $f$ in problem \eqref{Eq:problem} satisfies:
\begin{enumerate}
\item[(a)] $f$ is proper lower semicontinuous and $\rho$-weakly convex; and for any $\gamma\in (0,\rho^{-1})$, \textbf{either (b) or (c):}
\item[(b)] $f$ satisfies the \textbf{implicit Lipschitz subgradient} property with a constant $L_f>0$ (possibly depending on $\gamma$); or,
\item[(c)] $f$ satisfies the \textbf{implicit bounded subgradient} property with a constant $\hat{L}_f>0$ (possibly depending on $\gamma$).
\end{enumerate}
\end{assumption}

We do not assume the following hypotheses:
the strict complementarity condition used in \cite{Zhang-Luo18},  any rank assumption (such as $\mathrm{Im}(A)\subseteq \mathrm{Im}(\mathrm{B})$ when considering the two- (multi-)block case $Ax+By=0$) used in \cite{Wang19}, the linear independence constrained qualification (LICQ) used in \cite{Bertsekas82,Nocedal99} (implying the full-rank assumption in the linear constraint case).
Assumption \ref{Assump:MEAL} is mild as discussed in Section \ref{sc:implicit-property}.

According to \eqref{Eq:Moreauenvelope_AL} and the update \eqref{alg:MEAL} of MEAL, we have
\begin{align}
\label{Eq:stationary-MEAL}
\nabla \phi_{\beta_k}(z^k,\lambda^k)=
\left(
\begin{array}{c}
(\eta\gamma)^{-1} (z^k - z^{k+1})\\
\beta_k^{-1}(\lambda^{k+1}-\lambda^k)
\end{array}
\right)
\in
\left(
\begin{array}{c}
\partial f(x^{k+1}) + A^T\lambda^{k+1}\\
Ax^{k+1}-b
\end{array}
\right).
\end{align}
Let
\begin{align}
\label{Eq:measure-MEAL}
\xi_{\mathrm{meal}}^k := \min_{0\leq t \leq k} \|\nabla \phi_{\beta_t}(z^t,\lambda^t)\| , \ \forall k\in \mathbb{N}.
\end{align}
Then according to \eqref{Eq:stationary-MEAL}, the bound $\xi_{\mathrm{meal}}^k \leq \varepsilon$ implies
\begin{align*}
\min_{0\leq t \leq k} \mathrm{dist}\left\{0,\left(
\begin{array}{c}
\partial f(x^{t+1}) + A^T\lambda^{t+1}\\
Ax^{t+1}-b
\end{array}
\right)\right\} \leq \xi^k_{\mathrm{meal}} \leq \varepsilon,
\end{align*}
that is, MEAL achieves $\varepsilon$-accurate first-order stationarity for problem \eqref{Eq:problem} within $k$ iterations.
Hence, $\xi_{\mathrm{meal}}^k$ is a valid stationarity measure of MEAL.
Define iteration complexity:
\begin{align}
    \label{Eq:itercomplexity-meal}
    T_{\varepsilon} = \inf\left\{t\geq 1: \|\nabla \phi_{\beta_t}(z^t,\lambda^t)\| \leq \varepsilon \right\}.
\end{align}
Comparing $T_{\varepsilon}$ to the common iteration complexity
\begin{align*}
    \hat{T}_{\varepsilon}= \inf\left\{t\geq 1: \mathrm{dist}(0,\partial f(x^t)+A^T\lambda^t) \leq \epsilon \ \text{and}\ \|Ax^t-b\| \leq \varepsilon \right\},
\end{align*}
we get $T_{\varepsilon}\ge \hat{T}_{\varepsilon}$.

If $f$ is differentiable, $\mathrm{dist}(0,\partial f(x^t)+A^T\lambda^t)$ reduces to $\|\nabla f(x^t)+A^T\lambda^t\|$.

\subsection{Convergence Theorems of MEAL}
\label{sc:MEAL-convergence}

We present the quantities used to state the convergence results of MEAL.
Let
\begin{align}
\label{Eq:function-P}
{\cal P}_{\beta}(x,z,\lambda) = {\cal L}_{\beta}(x,\lambda) + \frac{1}{2\gamma}\|x-z\|^2,
\end{align}
for some $\beta, \gamma>0.$
Then according to \eqref{alg:MEAL-reformulation}, MEAL can be interpreted as a primal-dual update with respect to ${\cal P}_{\beta_k}(x,z,\lambda)$ at the $k$-th iteration, that is, updating $x^{k+1}$, $z^{k+1}$, and $\lambda^{k+1}$ by minimization, gradient descent, and gradient ascent respectively.

Based on \eqref{Eq:function-P}, we introduce the following \text{Lyapunov functions} for MEAL:
\begin{align}
\label{Eq:Lyapunov-seq-MEAL-S1}
{\cal E}_{\mathrm{meal}}^k
:= {\cal P}_{\beta_k}(x^k,z^k,\lambda^k)+ 2\alpha_k \|z^k - z^{k-1}\|^2, \ \forall k\geq 1,
\end{align}
associated with the \textit{implicit Lipschitz subgradient} assumption and
\begin{align}
\label{Eq:Lyapunov-seq-MEAL-S2}
\tilde{\cal E}_{\mathrm{meal}}^k := {\cal P}_{\beta_k}(x^k,z^k,\lambda^k) + 3\alpha_k \|z^k-z^{k-1}\|^2, \ \forall k\geq 1,
\end{align}
associated with the \textit{implicit bounded subgradient} assumption,
where
\begin{align}
\label{Eq:alphak}
\alpha_k := \frac{\beta_k+\beta_{k+1} + \gamma\eta(1-\eta/2) }{2c_{\gamma,A}\beta_k^2}, \ \forall k \in \mathbb{N},
\end{align}
and $c_{\gamma,A}:= \gamma^2 \tilde{\sigma}_{\min}(A^TA)$.
When $\beta$ is fixed, we also fix
\begin{align}
\label{Eq:alpha}
\alpha := \frac{2\beta+\gamma \eta(1-\eta/2)}{2c_{\gamma,A}\beta^2}.
\end{align}

\begin{theorem}[Iteration Complexity of MEAL]
\label{Theorem:Convergence-MEAL}
Suppose that Assumptions \ref{Assump:feasibleset} and \ref{Assump:MEAL}(a) hold. Pick $\gamma \in (0,\rho^{-1})$ and $\eta \in (0,2)$.
Let $\{(x^k,z^k,\lambda^k)\}$ be a sequence generated by MEAL \eqref{alg:MEAL-reformulation}.
The following claims hold:
\begin{enumerate}
\item[(a)]
Set $\beta$
sufficiently large such that in \eqref{Eq:alpha}, $\alpha < \min\left\{\frac{1-\gamma \rho}{4\gamma(1+\gamma L_f)^2},  \frac{1}{8\gamma}(\frac{2}{\eta}-1)\right\}$.
Under Assumption \ref{Assump:MEAL}(b), if $\{{\cal E}_{\mathrm{meal}}^k\}$ is lower bounded, then $\xi^k_{\mathrm{meal}} = o(1/\sqrt{k})$ for $\xi_{\mathrm{meal}}^k$ in \eqref{Eq:measure-MEAL}.

\item[(b)] Pick any $K\geq 1$. Set $\{\beta_k\}$ so that in \eqref{Eq:alphak},  $\alpha_k \equiv \frac{\alpha^*}{K}$ for some positive constant $\alpha^* \leq \min\left\{ \frac{1-\rho \gamma}{6\gamma}, \frac{1}{12\gamma}\left(\frac{2}{\eta} -1\right)\right\}$.
Under Assumption \ref{Assump:MEAL}(c), if $\{\tilde{\cal E}_{\mathrm{meal}}^k\}$ is lower bounded, then
$\xi_{\mathrm{meal}}^K \leq \tilde{c}_1/\sqrt{K}$ for some constant $\tilde{c}_1>0$.
\end{enumerate}
\end{theorem}

Section \ref{sc:discussion-boundedness} provides conditions sufficient for the lower-boundedness assumptions. Let us interpret the theorem.
To achieve an $\varepsilon$-accurate stationary point, the iteration complexity of MEAL is  $o(\varepsilon^{-2})$ assuming the implicit Lipschitz subgradient property and ${\cal O}(\varepsilon^{-2})$ assuming the implicit bounded subgradient property.
Both iteration complexities are consistent with the existing results of ${\cal O}(\varepsilon^{-2})$ in~\cite{Hajinezhad-Hong19,Hong17-Prox-PDA,Xie-Wright19,Zhang-Luo20}.
The established results of MEAL also hold for proximal ALM by setting $\eta = 1$.
We note that it is not our goal to pursue any better complexity (e.g., using momentum) in this paper.

\begin{remark}
Let $\bar{\alpha}:= \min\left\{\frac{1-\gamma \rho}{4\gamma(1+\gamma L_f)^2},\frac{1}{8\gamma}(\frac{2}{\eta}-1)\right\}$. By \eqref{Eq:alpha}, the requirement $0<\alpha < \bar{\alpha}$ in Theorem \ref{Theorem:Convergence-MEAL}(a) is met by setting
\begin{align}
\label{Eq:cond-beta-MEAL-S1}
\beta > \frac{1+\sqrt{1+\eta(2-\eta)\gamma c_{\gamma,A}\bar{\alpha}}}{2c_{\gamma,A}\bar{\alpha}}.
\end{align}
Similarly, 
the assumption $\alpha_k = \frac{\alpha^*}{K}$ in Theorem \ref{Theorem:Convergence-MEAL}(b) is met by setting
\begin{align}
\label{Eq:cond-beta-MEAL-S2}
\beta_k = \frac{K\left(1+\sqrt{1+\eta(2-\eta)\gamma c_{\gamma,A}\alpha^*/K}\right)}{2c_{\gamma,A}\alpha^*}, \ k=1,\ldots, K.
\end{align}
\end{remark}

Next, we establish global convergence (whole sequence convergence regardless of initial points) and its rate for MEAL under the K{\L} inequality (Definition \ref{Def-KLProp}).
Let
$\hat{z}^k := z^{k-1}$, $y^k:= (x^k,z^k,\lambda^k,\hat{z}^k), \ \forall k\geq 1,$
$y:= (x,z,\lambda,\hat{z}) \in \mathbb{R}^n \times \mathbb{R}^n \times \mathbb{R}^m \times \mathbb{R}^n,$ and
\begin{align}
\label{Eq:Lyapunov-fun-MEAL}
{\cal P}_{\mathrm{meal}}(y) := {\cal P}_{\beta}(x,z,\lambda) + 3 \alpha \|z-\hat{z}\|^2
\end{align}
where $\alpha$ is defined in \eqref{Eq:alpha}.

\begin{proposition}[Global convergence and rate of MEAL]
\label{Proposition:globalconv-MEAL}
Suppose that the assumptions required for Theorem \ref{Theorem:Convergence-MEAL}(a) hold and that $\{(x^k,z^k,\lambda^k)\}$ generated by MEAL \eqref{alg:MEAL-reformulation} is bounded.
If ${\cal P}_{\mathrm{meal}}$ satisfies the K{\L} property at some point $y^*:= (x^*,x^*,\lambda^*,x^*)$ with an exponent of $\theta \in [0,1)$, where $(x^*,\lambda^*)$ is a limit point of $\{(x^k,\lambda^k)\}$, then
\begin{enumerate}
\item[(a)] the whole sequence $\{\hat{y}^k:=(x^k,z^k,\lambda^k)\}$ converges to $\hat{y}^*:=(x^*,x^*,\lambda^*)$; and

\item[(b)] the following rate-of-convergence results hold: (1) if $\theta =0$, then $\{\hat{y}^k\}$ converges within a finite number of iterations; (2) if $\theta \in (0,\frac{1}{2}]$, then $\|\hat{y}^k- \hat{y}^*\| \leq c \tau^k$ for all $k\geq k_0$, for certain $k_0>0, c>0, \tau \in (0,1)$; and (3) if $\theta \in (\frac{1}{2},1)$, then $\|\hat{y}^k - \hat{y}^*\| \leq c k^{-\frac{1-\theta}{2\theta-1}}$ for all $k\geq k_0$, for certain $k_0>0, c>0$.
\end{enumerate}
\end{proposition}

In Proposition \ref{Proposition:globalconv-MEAL}, the K{\L} property of ${\cal P}_{\mathrm{meal}}$ defined in \eqref{Eq:Lyapunov-fun-MEAL}
plays a central role in the establishment of global convergence of MEAL. The K{\L} exponent determines the convergence speed of MEAL; particularly, the exponent $\theta=1/2$ implies linear convergence so it is most desirable. Below we give some results on $\theta$, which are obtainable from~\cite[page 43]{Shiota1997}, \cite[Theorem 3.1]{Bolte-KL2007a}, \cite[Lemma 5]{Zeng-BCD19}, and \cite[Theorem 3.6 and Corollary 5.2]{Li-Pong-KLexponent18}.

\begin{proposition}
\label{Propos:KL-property-Lyapunov}
The following claims hold:
\begin{enumerate}
\item[(a)] If $f$ is subanalytic with a closed domain and continuous on its domain, then ${\cal P}_{\mathrm{meal}}$ defined in \eqref{Eq:Lyapunov-fun-MEAL}  is a K{\L} function;
\item[(b)] If ${\cal L}_{\beta}(x,\lambda)$ defined in \eqref{Eq:augmented-Lagrangian} has the K{\L} property at some point $(x^*,\lambda^*)$ with exponent $\theta \in [1/2,1)$, then ${\cal P}_{\mathrm{meal}}$ has the K{\L} property at $(x^*,x^*,\lambda^*,x^*)$ with exponent $\theta$;
\item[(c)] If $f$ has the following form: 
\begin{align}
\label{Eq:f-KL-1/2}
f(x) = \min_{1\leq i\leq r}
\left\{ \frac{1}{2}x^TM_ix+u_i^Tx + c_i + P_i(x)\right\},
\end{align}
where $P_i$ are proper closed polyhedral functions, $M_i$ are symmetric matrices of size $n$, $u_i \in \mathbb{R}^n$ and $c_i \in \mathbb{R}$ for $i=1,\ldots,r$, then ${\cal L}_{\beta}$ is a K{\L} function with an exponent of $\theta =1/2$.
\end{enumerate}
\end{proposition}

Claim (a) can be obtained as follows.
The terms in ${\cal P}_{\mathrm{meal}}$
besides $f$ are polynomial functions, which are both real analytic and semialgebraic
\cite{Bochnak-semialgebraic1998}.
Since $f$ is subanalytic with a closed domain and continuous on its domain, by \cite[Lemma 5]{Zeng-BCD19}, ${\cal P}_{\mathrm{meal}}$ is also subanalytic with a closed domain and continuous on its domain.
By \cite[Theorem 3.1]{Bolte-KL2007a}, ${\cal P}_{\mathrm{meal}}$ is a K{\L} function.
Claim (b) can be verified by applying \cite[Theorem 3.6]{Li-Pong-KLexponent18} to ${\cal P}_{\mathrm{meal}}$.
Claim (c) can be established as follows.
The class of functions $f$ defined by \eqref{Eq:f-KL-1/2} are weakly convex with a modulus $\rho = 2 \max_{1\leq i\leq r} \|M_i\|$. According to \cite[Sec. 5.2]{Li-Pong-KLexponent18}, this class covers many nonconvex functions such as SCAD \cite{Fan-SCAD} and MCP \cite{Zhang-MCP} in statistical learning.
The function ${\cal L}_{\beta}(x,\lambda) = \frac{\beta}{2} \|Ax+\beta^{-1}\lambda-b\|^2 + (f(x) - \frac{1}{2\beta}\|\lambda\|^2)$.
according to \cite[Corollary 5.2]{Li-Pong-KLexponent18}, is a K{\L} function with an exponent of $1/2$.
More results on the K{\L} functions with exponent $1/2$ can be found in~\cite{Li-Pong-KLexponent18,Yu-Li-Pong-KLexponent21} and the references therein.

\subsection{Convergence of iMEAL}

When considering iMEAL, the Lyapunov functions need to be slightly modified into
\begin{align}
\label{Eq:Lyapunov-seq-iMEAL-S1}
{\cal E}_{\mathrm{imeal}}^k:= {\cal P}_{\beta_k}(x^k,z^k,\lambda^k)+ 3\alpha_k \|z^k-z^{k-1}\|^2, \ \forall k\geq 1,
\end{align}
associated with the implicit Lipschitz subgradient assumption,
and
\begin{align}
\label{Eq:Lyapunov-seq-iMEAL-S2}
\tilde{\cal E}_{\mathrm{imeal}}^k:= {\cal P}_{\beta_k}(x^k,z^k,\lambda^k)+ 4\alpha_k \|z^k-z^{k-1}\|^2, \ \forall k\geq 1,
\end{align}
associated with the implicit bounded subgradient assumption, where $\alpha_k$ is defined in \eqref{Eq:alphak}.

\begin{theorem}[Iteration Complexity of iMEAL]
\label{Theorem:Convergence-iMEAL}
Let Assumptions \ref{Assump:feasibleset} and \ref{Assump:MEAL}(a) hold, $\gamma \in (0,\rho^{-1})$, and $\eta \in (0,2)$.
Let $\{(x^k,z^k,\lambda^k)\}$ be a sequence generated by iMEAL \eqref{alg:iMEAL} with $\sum_{k=0}^{\infty} \epsilon_k^2 <\infty$.
The following claims hold:
\begin{enumerate}
\item[(a)]
Set $\beta$
sufficiently large
such that in \eqref{Eq:alpha},
$\alpha < \min\left\{\frac{1-\gamma \rho}{6\gamma(1+\gamma L_f)^2},  \frac{1}{12\gamma}(\frac{2}{\eta}-1)\right\}
$.
Under Assumption \ref{Assump:MEAL}(b), if $\{{\cal E}_{\mathrm{imeal}}^k\}$ is lower bounded, then  $\xi^k_{\mathrm{meal}} = o(1/\sqrt{k})$ (cf. \eqref{Eq:measure-MEAL}).

\item[(b)] Pick $K\geq 1$. Set $\{\beta_k\}$ such that in \eqref{Eq:alphak}, $\alpha_k\equiv\frac{\hat{\alpha}^*}{K}$ for some positive constant $\hat{\alpha}^* \leq \min\left\{ \frac{1-\rho \gamma}{8\gamma}, \frac{1}{16\gamma}(\frac{2}{\eta}-1) \right\}$.
Under Assumption \ref{Assump:MEAL}(c), if $\{\tilde{\cal E}_{\mathrm{imeal}}^k\}$ is lower bounded, then
$\xi_{\mathrm{meal}}^K \leq \tilde{c}_2/\sqrt{K}$ for some constant $\tilde{c}_2>0$.
\end{enumerate}
\end{theorem}

By Theorem \ref{Theorem:Convergence-iMEAL}, the iteration complexity of iMEAL is the same as that of MEAL and also consistent with that of inexact proximal ALM~\cite{Xie-Wright19} (when the stationary accuracy $\epsilon_k$ is square summable). 
Moreover, if the condition on $\epsilon_k$ is strengthened to be $\sum_{k=0}^{\infty} \epsilon_k <+\infty$ as required in the literature~\cite{Rockafellar76-PALM,Wang19}, then following a proof similar for Proposition \ref{Proposition:globalconv-MEAL}, global convergence and similar rates of MEAL also hold for iMEAL under the assumptions required for Theorem \ref{Theorem:Convergence-iMEAL}(a) and the K{\L} property.

\section{Convergence of LiMEAL for Composite Objective}
\label{sc:LiMEAL}

This section presents the convergence results of LiMEAL \eqref{alg:LiMEAL} for the constrained problem with a composite objective \eqref{Eq:problem-CP}. The proofs are postponed to Section \ref{sc:main-proofs} below.
Similar to Assumption \ref{Assump:MEAL}, we make the following assumptions.
\begin{assumption}
\label{Assump:LiMEAL}
The objective $f(x)=h(x)+g(x)$ in problem \eqref{Eq:problem-CP} satisfies:
\begin{enumerate}
\item[(a)] $h$ is differentiable and $\nabla h$ is Lipschitz continuous with a constant $L_h>0$;
\item[(b)] $g$ is proper lower-semicontinuous and $\rho_g$-weakly convex; and \textbf{either}
\item[(c)] $g$ has the \textbf{implicit Lipschitz subgradient} property with a constant $L_g>0$; \textbf{or}
\item[(d)] $g$ has the \textbf{implicit bounded subgradient} property with a constant $\hat{L}_g>0$.
\end{enumerate}
In (c) and (d), $L_g$ and $\hat{L}_g$ may depend on $\gamma$.
\end{assumption}

By the update \eqref{alg:LiMEAL} of LiMEAL, some simple derivations show that
\begin{align}
\label{Eq:xk+1-proxform-LiMEAL}
x^{k+1} = \mathrm{Prox}_{\gamma,g}(z^k - \gamma (\nabla h(x^k)+A^T\lambda^{k+1}))
\end{align}
and
\begin{align}
\label{Eq:stationary-LiMEAL}
g_{\mathrm{limeal}}^k:=
\left(
\begin{array}{c}
\gamma^{-1} (z^k - x^{k+1}) + (\nabla h(x^{k+1})-\nabla h(x^k))\\
\beta_k^{-1}(\lambda^{k+1}-\lambda^k)
\end{array}
\right)
\in
\left(
\begin{array}{c}
\partial f(x^{k+1}) + A^T\lambda^{k+1}\\
Ax^{k+1}-b
\end{array}
\right).
\end{align}
Actually, the term $\gamma^{-1}(z^k-x^{k+1})$ represents some \textit{prox-gradient sequence} frequently used in the analysis of algorithms for the unconstrained composite optimization (e.g., \cite{Davis-Drusvyatskiy19}).
Thus, let
\begin{align}
\label{Eq:measure-LiMEAL}
\xi_{\mathrm{limeal}}^k := \min_{0\leq t \leq k} \|g_{\mathrm{limeal}}^t\|, \ \forall k\in \mathbb{N},
\end{align}
which can be taken as an effective stationarity measure of LiMEAL for problem \eqref{Eq:problem-CP}.

In the following, we present the iteration complexity of LiMEAL for problem \eqref{Eq:problem-CP}. Since the prox-linear scheme is adopted in the update of $x^{k+1}$ in LiMEAL as described in \eqref{alg:LiMEAL}, thus, the proximal term (i.e., $\|x^k-x^{k-1}\|^2$) should be generally included in the associated Lyapunov functions of LiMEAL, shown as follows:
\begin{align}
\label{Eq:Lyapunov-seq-LiMEAL-S1}
{\cal E}^k_{\mathrm{limeal}}
&:= {\cal P}_{\beta_k}(x^k,z^k,\lambda^k) + 3\alpha_k (\gamma^2L_h^2 \|x^k-x^{k-1}\|^2 + \|z^k - z^{k-1}\|^2)
\end{align}
associated with the \textit{implicit Lipschitz gradient} assumption, and
\begin{align}
\label{Eq:Lyapunov-seq-LiMEAL-S2}
\tilde{\cal E}_{\mathrm{limeal}}^k:= {\cal P}_{\beta_k}(x^k,z^k,\lambda^k) + 4{\alpha}_{k}(\gamma^2L_h^2 \|x^k-x^{k-1}\|^2 + \|z^{k}-z^{k-1}\|^2),
\end{align}
associated with the \textit{implicit bounded subgradient} assumption,
where $\alpha_k$ is defined in \eqref{Eq:alphak}.

The iteration complexity of MEAL can be similarly generalized to LiMEAL as follows.

\begin{theorem}[Iteration Complexity of LiMEAL]
\label{Theorem:Convergence-LiMEAL}
Take Assumptions \ref{Assump:feasibleset} and \ref{Assump:LiMEAL}(a)-(b). Pick $\eta \in (0,2)$ and $0<\gamma<\frac{2}{(\rho_g+L_h)\left(1+\sqrt{1+\frac{2(2-\eta)\eta L_h^2}{(\rho_g+L_h)^2}} \right)}$.
Let $\{(x^k,z^k,\lambda^k)\}$ be a sequence generated by LiMEAL \eqref{alg:LiMEAL}.
The following claims hold:
\begin{enumerate}
\item[(a)] Set $\beta$
sufficiently large such that $\alpha < \min\left\{\frac{1}{12\gamma}(\frac{2}{\eta}-1), \frac{1-\gamma(\rho_g+L_h) - \eta(1-\eta/2)\gamma^2L_h^2}{6\gamma \left((1+\gamma L_g)^2 + \gamma^2 L_h^2 \right)}\right\}$.
Under Assumption \ref{Assump:LiMEAL}(c), if $\{{\cal E}^k_{\mathrm{limeal}}\}$ is lower bounded,
then
$\xi^k_{\mathrm{limeal}} = o(1/\sqrt{k})$.

\item[(b)]
Pick $K\geq 1$. Set $\{\beta_k\}$ such that $\alpha_k \equiv \frac{\bar{\alpha}^*}{K}$ for some positive constant $\bar{\alpha}^*\leq\min\Big\{\frac{1-\gamma\left(\rho_g+L_h)-\eta(1-\eta/2)\gamma^2 L_h^2 \right)}{8\gamma(1+\gamma^2L_h^2)} $, $\frac{1}{16\gamma}\left(\frac{2}{\eta} -1 \right)\Big\}$.
Under Assumption \ref{Assump:LiMEAL}(d), if $\{\tilde{\cal E}^k_{\mathrm{limeal}}\}$ is lower bounded,
then $\xi_{\mathrm{limeal}}^K \leq \tilde{c}_3/\sqrt{K}$ for some constant $\tilde{c}_3>0$.
\end{enumerate}
\end{theorem}

Similar to the discussions following Theorem \ref{Theorem:Convergence-MEAL}, to yield an $\varepsilon$-accurate first-order stationary point, the iteration complexity of LiMEAL is $o(\varepsilon^{-2})$ under the \textit{implicit Lipschitz subgradient} assumption and ${\cal O}(\varepsilon^{-2})$ under the \textit{implicit bounded subgradient} assumption,
as demonstrated by Theorem \ref{Theorem:Convergence-LiMEAL}. The conditions on $\beta$ and $\beta_k$ in these two cases can be derived similarly to \eqref{Eq:cond-beta-MEAL-S1} and \eqref{Eq:cond-beta-MEAL-S2}, respectively.

In the following, we establish the global convergence and rates of LiMEAL under assumptions required for Theorem \ref{Theorem:Convergence-LiMEAL}(a) and the K{\L} property.
Specifically, let
$
\hat{x}^k:= x^{k-1}, \ \hat{z}^k := z^{k-1}, \ {y}^k:= (x^k,z^k,\lambda^k,\hat{x}^k,\hat{z}^k), \ \forall k\geq 1,
$
${y}:= (x,z,\lambda,\hat{x},\hat{z}) \in \mathbb{R}^n \times \mathbb{R}^n \times \mathbb{R}^m \times \mathbb{R}^n\times \mathbb{R}^n,$ and
\begin{align}
\label{Eq:Lyapunov-fun-LiMEAL}
{\cal P}_{\mathrm{limeal}}({y}) := {\cal P}_{\beta}(x,z,\lambda) + 4\alpha \left(\|z-\hat{z}\|^2 + \gamma^2 L_h^2 \|x-\hat{x}\|^2\right).
\end{align}

\begin{proposition}[Global convergence and rates of LiMEAL]
\label{Proposition:globalconv-LiMEAL}
Suppose that Assumptions \ref{Assump:feasibleset} and \ref{Assump:LiMEAL}(a)-(c) hold and that the sequence $\{(x^k,z^k,\lambda^k)\}$ generated by LiMEAL \eqref{alg:LiMEAL} is bounded.
If $\gamma \in (0,\frac{1}{\rho_g+L_h})$, $\eta \in (0,2)$, $0<\alpha < \min \left\{\frac{1}{8\gamma}\left(\frac{2}{\eta}-1\right), \frac{1-\gamma(\rho_g+L_h)}{8\gamma \left((1+\gamma L_g)^2+\gamma^2 L_h^2 \right)} \right\}$,
and ${\cal P}_{\mathrm{limeal}}$ satisfies the K{\L} property at some point $y^*:= (x^*,x^*,\lambda^*,x^*,x^*)$ with an exponent of $\theta \in [0,1)$, where $(x^*,\lambda^*)$ is a limit point of $\{(x^k,\lambda^k)\}$, then
\begin{enumerate}
\item[(a)] the whole sequence $\{\hat{y}^k:=(x^k,z^k,\lambda^k)\}$ converges to $\hat{y}^*:=(x^*,x^*,\lambda^*)$; and

\item[(b)] all the rates of convergence results in Proposition \ref{Proposition:globalconv-MEAL}(b) also hold for LiMEAL.
\end{enumerate}
\end{proposition}

\begin{remark}
The established results in this section is more general than those in \cite{Zhang-Luo18} and done under weaker assumptions on $h$ and for more general class of $g$. Specifically, as discussed in Section \ref{sc:relation-existing-methods}, the algorithm studied in \cite{Zhang-Luo18} is a prox-linear version of LiMEAL with $g$ being an indicator function of a box constraint set.
In~\cite{Zhang-Luo18}, global convergence and a linear rate of {proximal inexact ALM} were proved for quadratic programming, where that the augmented Lagrangian satisfies the K{\L} inequality with exponent $1/2$.
Besides, the strict complementarity condition required in \cite{Zhang-Luo18} is also removed in this paper for LiMEAL.
\end{remark}

\section{Main Proofs}
\label{sc:main-proofs}

In this section, we first prove some lemmas and then present the proofs of our main convergence results.

\subsection{Preliminary Lemmas}
\label{sc:preliminary-lemmas}

\subsubsection{Lemmas on Iteration Complexity and Global Convergence}

The first lemma concerns the convergence speed of a nonenegative sequence $\{\xi_k\}$ satisfying the following relation
\begin{align}
\label{Eq:sequence-fixed}
\tilde{\eta} \xi_k^2 \leq ({\cal E}_k - {\cal E}_{k+1}) + \tilde{\epsilon}_k^2, \ \forall k\in \mathbb{N},
\end{align}
where $\tilde{\eta}>0$, $\{{\cal E}_k\}$ and $\{\tilde{\epsilon}_k\}$ are two nonnegative sequences, and $\sum_{k=1}^{\infty} \tilde{\epsilon}_k^2 < +\infty$.

\begin{lemma}
\label{Lemma:sequence-fixed}
For any sequence $\{\xi_k\}$ satisfying \eqref{Eq:sequence-fixed}, $\tilde{\xi}_k := \min_{1\leq t\leq k} \xi_t = o(1/\sqrt{k})$.
\end{lemma}

\begin{proof}
Summing \eqref{Eq:sequence-fixed} over $k$ from $1$ to $K$ and letting $K\rightarrow +\infty$ yields
\begin{align*}
\sum_{k=1}^{\infty} \xi_k^2 \leq \tilde{\eta}^{-1}\left({\cal E}_1 +  \sum_{k=1}^{\infty}\tilde{\epsilon}_k^2\right) <+\infty,
\end{align*}
which implies the desired convergence speed by $\frac{k}{2} \tilde{\xi}_k^2 \leq \sum_{\frac{k}{2}\leq j \leq k} {\xi}_j^2 \rightarrow 0$ as $k\rightarrow \infty$, as proved in \cite[Lemma 1.1]{Deng-parallelADMM17}.
\end{proof}

Then we provide a lemma to show the convergence speed of a nonenegative sequence $\{\xi_k\}$ satisfying the following relation instead of \eqref{Eq:sequence-fixed}
\begin{align}
\label{Eq:sequence-varying}
\tilde{\eta} \xi_k^2 \leq ({\cal E}_k - {\cal E}_{k+1}) + \tilde{\epsilon}_k^2 + {\alpha}_k \tilde{L}, \ \forall k\in \mathbb{N},
\end{align}
where $\tilde{\eta}>0,$ $\tilde{L}>0$, $\{{\cal E}_k\}$, $\{{\alpha}_k\}$ and $\{\tilde{\epsilon}_k\}$ are nonnegative sequences, and $\sum_{k=1}^{\infty} \tilde{\epsilon}_k^2 < +\infty$.

\begin{lemma}
\label{Lemma:sequence-varying}
Pick $K \ge 1$. Let $\{\xi_k\}$ be a nonnegative sequence satisfying \eqref{Eq:sequence-varying}. Set $\alpha_k \equiv \frac{\tilde{\alpha}}{K}$ for some $\tilde{\alpha}>0$. Then
$\tilde{\xi}_K := \min_{1\leq k\leq K} \xi_k \leq \tilde{c}/\sqrt{K}$ for some constant $\tilde{c}>0$.
\end{lemma}
\begin{proof}
Summing \eqref{Eq:sequence-varying} over $k$ from $1$ to $K$ yields
\begin{align*}
 \sum_{k=1}^K  \xi_k^2 \leq \frac{{\cal E}_1 + \sum_{k=1}^K \tilde{\epsilon}_k^2 + \tilde{L}\sum_{k=1}^K \alpha_k}{ \tilde{\eta}}.
\end{align*}
From $\sum_{k=1}^{\infty} \tilde{\epsilon}_k^2 < +\infty$ and  $\sum_{k=1}^K{\alpha}_k=\tilde{\alpha}$, we get $K\tilde{\xi}_K^2\le \sum_{k=1}^K  \xi_k^2\le\frac{{\cal E}_1 + \sum_{k=1}^\infty \tilde{\epsilon}_k^2 + \tilde{L}\tilde{\alpha}}{\tilde{\eta}}<+\infty$.
The result follows with $\tilde{c} := \sqrt{{\cal E}_1 + \sum_{k=1}^\infty \tilde{\epsilon}_k^2 + \tilde{L}\tilde{\alpha}}/\sqrt{{\tilde{\eta}}}$.
\end{proof}

In both Lemmas \ref{Lemma:sequence-fixed} and \ref{Lemma:sequence-varying}, the nonnegative assumption on the sequence $\{{\cal E}_k\}$ can be relaxed to its lower boundedness.

The following lemma presents the global convergence and rate of a sequence generated by some algorithm for the nonconvex optimization problem, based on the Kurdyka-{\L}ojasiewicz inequality, where the global convergence result is from \cite[Theorem 2.9]{Attouch13} while the rate results are from \cite[Theorem 5]{Attouch-Bolte09}.
\begin{lemma}[Existing global convergence and rate]
\label{Lemma:existing-global-converg}
Let ${\cal L}$ be a proper, lower semicontinuous function, and $\{u^k\}$ be a sequence that satisfies the following three conditions:
\begin{enumerate}
\item[(P1)]
(\textit{Sufficient decrease condition}) there exists a constant $a_1>0$ such that
${\cal L}(u^{k+1}) + a_1 \|u^{k+1}-u^k\|^2 \leq {\cal L}(u^{k}), \ \forall k\in \mathbb{N};$

\item[(P2)] (\textit{Bounded subgradient condition}) for each $k\in \mathbb{N}$, there exists $v^{k+1} \in \partial {\cal L}(u^{k+1})$ such that $\|v^{k+1}\| \leq a_2\|u^{k+1}-u^k\|$ for some constant $a_2>0$;

\item[(P3)] (\textit{Continuity condition}) there exist a subsequence $\{u^{k_j}\}$ and $\tilde{u}$ such that $u^{k_j} \rightarrow \tilde{u}$ and ${\cal L}(u^{k_j}) \rightarrow {\cal L}(\tilde{u})$ as $j\rightarrow \infty$.
\end{enumerate}
If ${\cal L}$ satisfies the K{\L} inequality at $\tilde{u}$ with an exponent of $\theta$, then
\begin{enumerate}
\item[(1)] $\{u^k\}$ converges to $\tilde{u}$; and

\item[(2)] depending on $\theta$,
(i) if $\theta =0$, then $\{u^k\}$ converges within a finite number of iterations; (ii) if $\theta \in (0,\frac{1}{2}]$, then $\|u^k- \tilde{u}\| \leq c \tau^k$ for all $k\geq k_0$, for certain $k_0>0, c>0, \tau \in (0,1)$; and (iii) if $\theta \in (\frac{1}{2},1)$, then $\|u^k - \tilde{u}\| \leq c k^{-\frac{1-\theta}{2\theta-1}}$ for all $k\geq k_0$, for certain $k_0>0, c>0$.
\end{enumerate}
\end{lemma}

\subsubsection{Lemmas on controlling dual ascent by primal descent}

In the following, we establish several lemmas to show that the dual ascent quantities of proposed algorithms can be controlled by the primal descent quantities.

\begin{lemma}[MEAL: controlling dual by primal]
\label{Lemma:dual-control-primal-MEAL}
Let $\{(x^k,z^k,\lambda^k)\}$ be a sequence generated by MEAL \eqref{alg:MEAL-reformulation}. Take $\gamma \in (0,\rho^{-1})$.
\begin{enumerate}
\item[(a)]
Under Assumptions \ref{Assump:feasibleset}, \ref{Assump:MEAL}(a), and \ref{Assump:MEAL}(b), we have for any $k\geq 1$,
\begin{align}
&\|A^T(\lambda^{k+1}-\lambda^{k})\| \leq (L_f+\gamma^{-1})\|x^{k+1}-x^k\|+\gamma^{-1}\|z^k-z^{k-1}\|, \label{Eq:A-lambda-MEAL}\\
&\|\lambda^{k+1}-\lambda^{k}\|^2 \leq 2c_{\gamma,A}^{-1}\left[(\gamma L_f+1)^2\|x^{k+1}-x^k\|^2+\|z^k-z^{k-1}\|^2\right], \label{Eq:lambda-MEAL}
\end{align}
where $c_{\gamma,A} = \gamma^2 \tilde{\sigma}_{\min}(A^TA)$.

\item[(b)]
Alternatively, under Assumptions \ref{Assump:feasibleset}, \ref{Assump:MEAL}(a), and \ref{Assump:MEAL}(c), we have for any $k\geq 1$,
\begin{align}
\|\lambda^{k+1}-\lambda^{k}\|^2 \leq 3c_{\gamma,A}^{-1}\left[4\gamma^2 \hat{L}^2_f + \|x^{k+1}-x^k\|^2+\|z^k-z^{k-1}\|^2 \right]. \label{Eq:lambda-MEAL-S2}
\end{align}
\end{enumerate}
\end{lemma}

\begin{proof}
The update \eqref{alg:MEAL-reformulation} of $x^{k+1}$ implies
\begin{align*}
x^{k+1} = \argmin_x \left\{f(x)+\langle \lambda^{k}, Ax-b \rangle + \frac{\beta_k}{2}\|Ax-b\|^2 + \frac{1}{2\gamma} \|x-z^k\|^2 \right\}.
\end{align*}
Its optimality condition and the update \eqref{alg:MEAL-reformulation} of $\lambda^{k+1}$ in MEAL together give us
\begin{align}
\label{Eq:optcond-x}
0 \in \partial \left(f+\frac{1}{2\gamma}\|\cdot-(z^k-\gamma A^T\lambda^{k+1})\|^2\right)(x^{k+1}).
\end{align}
Let $w^{k+1}:= z^k-\gamma A^T\lambda^{k+1}, \ \forall k\in \mathbb{N}.$
The above inclusion implies
\begin{align}
\label{Eq:xk+1-prox-wk+1}
x^{k+1} = \mathrm{Prox}_{\gamma,f}(w^{k+1}),
\end{align}
and thus by \eqref{Eq:Moreau-gradient},
\begin{align}
\label{Eq:A-lambda}
A^T\lambda^{k+1}
&= -\nabla {\cal M}_{\gamma,f}(w^{k+1})-\gamma^{-1}(x^{k+1}-z^k),
\end{align}
which further implies
\begin{align*}
\|A^T(\lambda^{k+1}-\lambda^{k})\|
&=\|(\nabla {\cal M}_{\gamma,f}(w^{k+1})-\nabla {\cal M}_{\gamma,f}(w^{k})) + \gamma^{-1}(x^{k+1}-x^k) - \gamma^{-1}(z^k-z^{k-1})\|.
\end{align*}

\textbf{(a)}
With Assumption \ref{Assump:MEAL}(b), the above equality yields
\begin{align*}
\|A^T(\lambda^{k+1}-\lambda^{k})\| \leq (L_f+\gamma^{-1})\|x^{k+1}-x^k\|+\gamma^{-1}\|z^k-z^{k-1}\|,
\end{align*}
which leads to \eqref{Eq:A-lambda-MEAL}.
By Assumption \ref{Assump:feasibleset} and the relation $\lambda^{k+1}-\lambda^k = \beta_k(Ax^{k+1}-b)$, $(\lambda^{k+1}-\lambda^k) \in \mathrm{Im}(A)$. Thus, from the above inequality, we deduce
\begin{align*}
\|\lambda^{k+1}-\lambda^{k}\| \leq \tilde{\sigma}_{\min}^{-1/2}(A^TA)\left[(L_f+\gamma^{-1})\|x^{k+1}-x^k\|+\gamma^{-1}\|z^k-z^{k-1}\| \right],
\end{align*}
and, further by $(u+v)^2 \leq 2(u^2+v^2)$ for any $u,v\in \mathbb{R}$,
\begin{align*}
\|\lambda^{k+1}-\lambda^{k}\|^2 \leq 2\tilde{\sigma}_{\min}^{-1}(A^TA)\left[(L_f+\gamma^{-1})^2\|x^{k+1}-x^k\|^2+\gamma^{-2}\|z^k-z^{k-1}\|^2\right].
\end{align*}

\textbf{(b)}
From Assumption \ref{Assump:MEAL}(c), we have
\begin{align*}
\|A^T(\lambda^{k+1}-\lambda^{k})\| \leq 2\hat{L}_f+\gamma^{-1}(\|x^{k+1}-x^k\|+\|z^k-z^{k-1}\|),
\end{align*}
which implies
\begin{align*}
\|\lambda^{k+1}-\lambda^{k}\| \leq \tilde{\sigma}_{\min}^{-1/2}(A^TA)\left[2\hat{L}_f+\gamma^{-1}(\|x^{k+1}-x^k\|+\|z^k-z^{k-1}\|) \right],
\end{align*}
and further by $(a+c+d)^2 \leq 3(a^2+c^2+d^2)$ for any $a,c,d\in \mathbb{R}$,
\begin{align*}
\|\lambda^{k+1}-\lambda^{k}\|^2 \leq 3\tilde{\sigma}_{\min}^{-1}(A^TA)\left[4\hat{L}^2_f + \gamma^{-2}(\|x^{k+1}-x^k\|^2+\|z^k-z^{k-1}\|^2)\right].
\end{align*}
\end{proof}

The similar lemma also holds for iMEAL shown as follows.
\begin{lemma}[iMEAL: controlling dual by primal]
\label{Lemma:dual-control-primal-iMEAL}
Let $(x^k,z^k,\lambda^k)$ be a sequence generated by iMEAL \eqref{alg:iMEAL}. Take $\gamma \in (0,\rho^{-1})$.
\begin{enumerate}
\item[(a)]
Under  Assumptions \ref{Assump:feasibleset} \ref{Assump:MEAL}(a), and \ref{Assump:MEAL}(b), for any $k\geq 1$,
\begin{align*}
\|\lambda^{k+1}-\lambda^{k}\|^2 \leq 3 c_{\gamma,A}^{-1}\left[(\gamma L_f+1)^2\|x^{k+1}-x^k\|^2+\|z^k-z^{k-1}\|^2 + \gamma^2(\epsilon_k + \epsilon_{k-1})^2\right].
\end{align*}

\item[(b)]
Alternatively, under Assumptions \ref{Assump:feasibleset} \ref{Assump:MEAL}(a), and \ref{Assump:MEAL}(c), for any $k\geq 1$,
\begin{align*}
\|\lambda^{k+1}-\lambda^{k}\|^2 \leq 4c_{\gamma,A}^{-1}\left[4\gamma^2\hat{L}^2_f + \|x^{k+1}-x^k\|^2+\|z^k-z^{k-1}\|^2 + \gamma^2(\epsilon_k + \epsilon_{k-1})^2\right].
\end{align*}
\end{enumerate}
\end{lemma}

\begin{proof}
The proof is similar to that of Lemma \ref{Lemma:dual-control-primal-MEAL}, but with \eqref{Eq:optcond-x} being replaced by
\begin{align*}
0 \in \partial \left(f+\frac{1}{2\gamma}\left\|\cdot-\Big(z^k-\gamma (A^T\lambda^{k+1}-s^k)\Big)\right\|^2\right)(x^{k+1}),
\end{align*}
and thus $w^{k+1}:= z^k-\gamma (A^T\lambda^{k+1}-s^k).$
\end{proof}

\begin{lemma}[LiMEAL: controlling dual by primal]
\label{Lemma:dual-control-primal-LiMEAL}
Let $\{(x^k,z^k,\lambda^k)\}$ be a sequence generated by LiMEAL \eqref{alg:LiMEAL}.
Take $\gamma \in (0,\rho_g^{-1})$.
\begin{enumerate}
\item[(a)]
Under Assumptions \ref{Assump:feasibleset}, \ref{Assump:LiMEAL}(a)-(b), and \ref{Assump:LiMEAL}(c), for any $k\geq 1$,
\begin{align}
&\|A^T(\lambda^{k+1}-\lambda^{k})\| \label{Eq:A-lambda-LiMEAL}\\
&\leq (L_g+\gamma^{-1})\|x^{k+1}-x^k\|+L_h\|x^k-x^{k-1}\|+\gamma^{-1}\|z^k-z^{k-1}\|, \nonumber\\
&\|\lambda^{k+1}-\lambda^{k}\|^2 \label{Eq:lambda-LiMEAL}\\
&\leq 3c_{\gamma,A}^{-1}\left[(\gamma L_g+1)^2\|x^{k+1}-x^k\|^2+\gamma^2 L_h^2\|x^k-x^{k-1}\|^2+ \|z^k-z^{k-1}\|^2\right]. \nonumber
\end{align}

\item[(b)]
Alternatively, under Assumptions \ref{Assump:feasibleset}, \ref{Assump:LiMEAL}(a)-(b), and \ref{Assump:LiMEAL}(d), for any $k\geq 1$,
\begin{align}
&\|\lambda^{k+1}-\lambda^{k}\|^2 \label{Eq:lambda-LiMEAL-S2}\\
&\leq 4c_{\gamma,A}^{-1}\left[4\gamma^2\hat{L}^2_g + \|x^{k+1}-x^k\|^2+ \gamma^2 L_h^2 \|x^k-x^{k-1}\|^2 +\|z^k-z^{k-1}\|^2\right]. \nonumber
\end{align}
\end{enumerate}
\end{lemma}

\begin{proof}
The proof is also similar to that of Lemma \ref{Lemma:dual-control-primal-MEAL}, but \eqref{Eq:optcond-x} needs to be modified to
\begin{align*}
0 \in \partial \left(g+\frac{1}{2\gamma}\|\cdot-\Big(z^k-\gamma (A^T\lambda^{k+1}+\nabla h(x^k))\Big)\|^2\right)(x^{k+1}),
\end{align*}
and thus $w^{k+1}:= z^k-\gamma (A^T\lambda^{k+1}+\nabla h(x^k))).$
\end{proof}

\subsubsection{Lemmas on One-step Progress}

Here, we provide several lemmas to characterize the progress achieved by a single iterate of the proposed algorithms.
\begin{lemma}[MEAL: one-step progress]
\label{Lemma:1-step-progress-MEAL}
Let $\{(x^k,z^k,\lambda^k)\}$ be a sequence generated by MEAL \eqref{alg:MEAL}. Take Assumption \ref{Assump:MEAL}(a), $\gamma \in (0,\rho^{-1})$, and $\eta\in (0,2)$. Then for any $k\in \mathbb{N}$,
\begin{align}
\label{Eq:1-step-progress-MEAL}
&{\cal P}_{\beta_k}(x^{k},z^k,\lambda^k) - {\cal P}_{\beta_{k+1}}(x^{k+1},z^{k+1},\lambda^{k+1})
\geq \frac{(1-\gamma\rho)}{2\gamma}\|x^{k+1}-x^k\|^2 \\
&+ \frac{1}{4\gamma}(\frac{2}{\eta}-1)\|z^{k+1}-z^k\|^2
+\frac{1}{4}\gamma \eta(2-\eta)\|\nabla \phi_{\beta_k}(z^k,\lambda^k)\|^2
-\alpha_k c_{\gamma,A}\|\lambda^{k+1}-\lambda^k\|^2, \nonumber
\end{align}
where $\alpha_k$ is presented in \eqref{Eq:alphak}
and $c_{\gamma,A} = \gamma^2 \tilde{\sigma}_{\min}(A^TA)$.
\end{lemma}

\begin{proof}
By the update \eqref{alg:MEAL-reformulation} of $x^{k+1}$ in MEAL, $x^{k+1}$ is updated via minimizing a strongly convex function ${\cal P}_{\beta_k} (x,z^k,\lambda^{k})$ with modulus at least $(\gamma^{-1}-\rho)$, we have
\begin{align}
\label{Eq:x-descent}
{\cal P}_{\beta_k}(x^k,z^k,\lambda^{k}) - {\cal P}_{\beta_k}(x^{k+1},z^k,\lambda^{k}) \geq \frac{\gamma^{-1}-\rho}{2} \|x^{k+1}-x^k\|^2. 
\end{align}
Next, recall in \eqref{alg:MEAL-reformulation}, $z^{k+1} = z^k + \eta(x^{k+1}-z^k)$ implies
\begin{align}
\label{Eq:zk+zk+1}
2x^{k+1}-z^k-z^{k+1} = (2\eta^{-1}-1)(z^{k+1}-z^k).
\end{align}
So we have
\begin{align*}
&{\cal P}_{\beta_k}(x^{k+1},z^k,\lambda^{k}) - {\cal P}_{\beta_k}(x^{k+1},z^{k+1},\lambda^{k})
=\frac{1}{2\gamma}(\|x^{k+1}-z^k\|^2 - \|x^{k+1}-z^{k+1}\|^2)\\
&=\frac{1}{2\gamma} \langle z^{k+1}-z^k, 2x^{k+1}-z^k-z^{k+1}\rangle
=\frac{1}{2\gamma}(\frac{2}{\eta}-1)\|z^{k+1}-z^k\|^2.
\end{align*}
Moreover, by the update $\lambda^{k+1}=\lambda^k +\beta_k(Ax^{k+1}-b)$, we have
\begin{align*}
{\cal P}_{\beta_k}(x^{k+1},z^{k+1},\lambda^k) - {\cal P}_{\beta_k}(x^{k+1},z^{k+1},\lambda^{k+1}) = -\beta_k^{-1}\|\lambda^{k+1}-\lambda^k\|^2,
\end{align*}
and
\begin{align*}
{\cal P}_{\beta_k}(x^{k+1},z^{k+1},\lambda^{k+1}) - {\cal P}_{\beta_{k+1}}(x^{k+1},z^{k+1},\lambda^{k+1}) = \frac{\beta_k-\beta_{k+1}}{2\beta_k^2}\|\lambda^{k+1}-\lambda^k\|^2.
\end{align*}
Combining the above four terms of estimates yields
\begin{align}
\label{Eq:primal-descent-MEAL}
&{\cal P}_{\beta_k}(x^{k},z^k,\lambda^k) - {\cal P}_{\beta_{k+1}}(x^{k+1},z^{k+1},\lambda^{k+1}) \\
&\geq \frac{(1-\rho\gamma)}{2\gamma}\|x^{k+1}-x^k\|^2 + \frac{1}{2\gamma}(\frac{2}{\eta}-1)\|z^{k+1}-z^k\|^2 -\frac{\beta_k+\beta_{k+1}}{2\beta_k^2}\|\lambda^{k+1}-\lambda^k\|^2. \nonumber
\end{align}
Then, we establish \eqref{Eq:1-step-progress-MEAL} from \eqref{Eq:primal-descent-MEAL}.
By the definition \eqref{Eq:stationary-MEAL} of $\nabla \phi_{\beta_k}(z^k,\lambda^k)$, we have
\begin{align*}
\|\nabla \phi_{\beta_k}(z^k,\lambda^k)\|^2 = (\eta \gamma)^{-2} \|z^k - z^{k+1}\|^2 + \beta_k^{-2} \|\lambda^{k+1}-\lambda^k\|^2,
\end{align*}
which implies
\begin{align*}
(\eta \gamma)^{-2} \|z^k - z^{k+1}\|^2 = \|\nabla \phi_{\beta_k}(z^k,\lambda^k)\|^2 - \beta_k^{-2} \|\lambda^{k+1}-\lambda^k\|^2.
\end{align*}
Substituting this into the above inequality yields
\begin{align*}
&{\cal P}_{\beta_k}(x^{k},z^k,\lambda^k) - {\cal P}_{\beta_{k+1}}(x^{k+1},z^{k+1},\lambda^{k+1})
\geq \frac{(1-\gamma\rho)}{2\gamma}\|x^{k+1}-x^k\|^2 \nonumber \\
&+ \frac{1}{4\gamma}(\frac{2}{\eta}-1)\|z^{k+1}-z^k\|^2
+\frac{1}{4}\gamma \eta(2-\eta)\|\nabla \phi_{\beta_k}(z^k,\lambda^k)\|^2
-\alpha_k c_{\gamma,A}\|\lambda^{k+1}-\lambda^k\|^2,
\end{align*}
where $\alpha_k = \frac{\beta_k+\beta_{k+1}+\gamma \eta(1-\eta/2)}{2c_{\gamma,A}\beta_k^2}$.
This finishes the proof.
\end{proof}

Next, we provide a lemma for iMEAL \eqref{alg:iMEAL}.
\begin{lemma}[iMEAL: one-step progress]
\label{Lemma:1-step-progress-iMEAL}
Let $\{(x^k,z^k,\lambda^k)\}$ be a sequence generated by iMEAL \eqref{alg:iMEAL}.
Take Assumptions \ref{Assump:MEAL}(a) and (b), $\gamma \in (0,\rho^{-1})$, and $\eta \in (0,2)$. It holds that
\begin{align}
\label{Eq:1-step-progress-iMEAL}
&{\cal P}_{\beta_k}(x^{k},z^k,\lambda^k) - {\cal P}_{\beta_{k+1}}(x^{k+1},z^{k+1},\lambda^{k+1})\\
&\geq \frac{(1-\gamma\rho)}{2\gamma}\|x^{k+1}-x^k\|^2 +\langle s^k, x^k - x^{k+1} \rangle
+ \frac{1}{4\gamma}(\frac{2}{\eta}-1)\|z^{k+1}-z^k\|^2 \nonumber \\
&+\frac{1}{2}\gamma \eta(1-\eta/2)\|\nabla \phi_{\beta_k}(z^k,\lambda^k)\|^2
-\alpha_k c_{\gamma,A}\|\lambda^{k+1}-\lambda^k\|^2, \ \forall k\in \mathbb{N}. \nonumber
\end{align}
\end{lemma}

\begin{proof}
The proof of this lemma is similar to that of Lemma \ref{Lemma:1-step-progress-MEAL} and uses the descent quantity along the update of $x^{k+1}$.
By the update \eqref{alg:iMEAL} of $x^{k+1}$ in iMEAL and noticing that ${\cal L}_{\beta_k}(x,\lambda^k)+ \frac{\|x-z^k\|}{2\gamma}$ is strongly convex with modulus at least $(\gamma^{-1}-\rho)$, we have
\begin{align*}
{\cal P}_{\beta_k}(x^k,z^k,\lambda^k)
& \geq {\cal P}_{\beta_k}(x^{k+1},z^k,\lambda^k)+\langle s^k, x^k - x^{k+1} \rangle+ \frac{\gamma^{-1}-\rho}{2}\|x^{k+1}-x^k\|^2.
\end{align*}
By replacing \eqref{Eq:x-descent} in the proof of Lemma \ref{Lemma:1-step-progress-MEAL} with the above inequality and following the rest part of its proof, we obtain the following inequality
\begin{align*}
&{\cal P}_{\beta_k}(x^{k},z^k,\lambda^k) - {\cal P}_{\beta_{k+1}}(x^{k+1},z^{k+1},\lambda^{k+1})
\geq \frac{1-\gamma\rho}{2\gamma}\|x^{k+1}-x^k\|^2 +\langle s^k, x^k - x^{k+1} \rangle \nonumber \\
&+ \frac{1}{2\gamma}(\frac{2}{\eta}-1)\|z^{k+1}-z^k\|^2 -\frac{\beta_k+\beta_{k+1}}{2\beta_k^2}\|\lambda^{k+1}-\lambda^k\|^2.
\end{align*}
We can establish \eqref{Eq:1-step-progress-iMEAL} with a derivation similar to that in the proof of Lemma \ref{Lemma:1-step-progress-MEAL}.
\end{proof}

Also, we state a similar lemma for one-step progress of LiMEAL \eqref{alg:LiMEAL} as follows.
\begin{lemma}[LiMEAL: one-step progress]
\label{Lemma:1-step-progress-LiMEAL}
Let $\{(x^k,z^k,\lambda^k)\}$ be a sequence generated by LiMEAL \eqref{alg:LiMEAL}.
Take Assumptions \ref{Assump:LiMEAL}(a) and (b), $\gamma \in (0,\rho_g^{-1})$, and $\eta \in (0,2)$. We have
\begin{align}
\label{Eq:1-step-progress-LiMEAL}
&{\cal P}_{\beta_k}(x^{k},z^k,\lambda^k) - {\cal P}_{\beta_{k+1}}(x^{k+1},z^{k+1},\lambda^{k+1})\\
&\geq \left(\frac{1-\gamma(\rho_g+L_h)}{2\gamma} - \frac{1}{4}\gamma (2-\eta)\eta L_h^2\right)\|x^{k+1}-x^k\|^2 \nonumber\\
&+ \frac{1}{4\gamma}(\frac{2}{\eta}-1)\|z^{k+1}-z^k\|^2
+ \frac{1}{4}\gamma (1-\eta_k/2)\eta \|g_{\mathrm{limeal}}^k\|^2- \alpha_k c_{\gamma,A}\|\lambda^{k+1}-\lambda^k\|^2, \ \forall k\in \mathbb{N}. \nonumber
\end{align}
\end{lemma}

\begin{proof}
The proof of this lemma is similar to that of Lemma \ref{Lemma:1-step-progress-MEAL}.
By the update \eqref{alg:LiMEAL} of $x^{k+1}$ in LiMEAL,
$x^{k+1}$ is updated via minimizing $(\gamma^{-1}-\rho_g)$-strongly convex ${\cal L}_{\beta_k,f^k}(x,\lambda^k)+ \frac{\|x-z^k\|}{2\gamma}$, so
\begin{align*}
{\cal L}_{\beta_k,f^k}(x^k,\lambda^k)+ \frac{\|x^k-z^k\|^2}{2\gamma}
\geq {\cal L}_{\beta_k,f^k}(x^{k+1},\lambda^k)+ \frac{\|x-z^k\|^2}{2\gamma} + \frac{\gamma^{-1}-\rho_g}{2}\|x^{k+1}-x^k\|^2.
\end{align*}
By definition, ${\cal L}_{\beta_k,f^k}(x,\lambda) = h(x^k)+\langle \nabla h(x^k),x-x^k\rangle + g(x) + \langle \lambda, Ax-b\rangle + \frac{\beta}{2}\|Ax-b\|^2$ and ${\cal P}_{\beta_k}(x,z,\lambda) = h(x)+g(x)+ \langle \lambda, Ax-b\rangle + \frac{\beta_k}{2}\|Ax-b\|^2+\frac{\|x-z\|^2}{2\gamma}$, so the above inequality implies
\begin{align*}
{\cal P}_{\beta_k}(x^k,z^k,\lambda^k)
& \geq {\cal P}_{\beta_k}(x^{k+1},z^k,\lambda^k)+ \frac{\gamma^{-1}-\rho_g}{2}\|x^{k+1}-x^k\|^2 \nonumber\\
&- (h(x^{k+1}) - h(x^k) - \langle \nabla h(x^k),x^{k+1}-x^k\rangle) \nonumber\\
&\geq {\cal P}_{\beta_k}(x^{k+1},z^k,\lambda^k)+ \frac{\gamma^{-1}-\rho_g-L_h}{2}\|x^{k+1}-x^k\|^2,
\end{align*}
where the second inequality is due to the $L_h$-Lipschitz continuity of $\nabla h$.
By replacing \eqref{Eq:x-descent} in the proof of Lemma \ref{Lemma:1-step-progress-MEAL} with the above inequality and following the rest part of that proof, we obtain
\begin{align}
&{\cal P}_{\beta_k}(x^{k},z^k,\lambda^k) - {\cal P}_{\beta_{k+1}}(x^{k+1},z^{k+1},\lambda^{k+1}) \label{Eq:primal-descent-LiMEAL}\\
&\geq \frac{1-\gamma(\rho_g+L_h)}{2\gamma}\|x^{k+1}-x^k\|^2 + \frac{1}{2\gamma}(\frac{2}{\eta}-1)\|z^{k+1}-z^k\|^2 -\frac{\beta_k+\beta_{k+1}}{2\beta_k^2}\|\lambda^{k+1}-\lambda^k\|^2. \nonumber
\end{align}

Next, based on the above inequality, we establish \eqref{Eq:1-step-progress-LiMEAL}.
By the definition \eqref{Eq:stationary-LiMEAL} of $g_{\mathrm{limeal}}^k$ and noticing that $z^k-x^{k+1} = -\eta^{-1}(z^{k+1}-z^k) $ by the update \eqref{alg:LiMEAL} of $z^{k+1}$, we have
\[
\|g_{\mathrm{limeal}}^k\|^2 \leq 2L_h^2 \|x^{k+1}-x^k\|^2 + 2 (\gamma\eta)^{-2}\|z^{k+1}-z^k\|^2 + \beta_k^{-2}\|\lambda^{k+1}-\lambda^k\|^2,
\]
which implies
\[
(\gamma\eta)^{-2}\|z^{k+1}-z^k\|^2 \geq \frac{1}{2} \|g_{\mathrm{limeal}}^k\|^2 - \frac{1}{2}\beta_k^{-2}\|\lambda^{k+1}-\lambda^k\|^2 - L_h^2 \|x^{k+1}-x^k\|^2.
\]
Substituting this inequality into \eqref{Eq:primal-descent-LiMEAL} yields
\begin{align*}
&{\cal P}_{\beta_k}(x^{k},z^k,\lambda^k) - {\cal P}_{\beta_{k+1}}(x^{k+1},z^{k+1},\lambda^{k+1})\\
&\geq \left(\frac{1-\gamma(\rho_g+L_h)}{2\gamma} - \frac{1}{4}\gamma (2-\eta)\eta L_h^2\right)\|x^{k+1}-x^k\|^2 \\
&+ \frac{1}{4\gamma}(\frac{2}{\eta}-1)\|z^{k+1}-z^k\|^2
+ \frac{1}{4}\gamma (1-\eta/2)\eta \|g_{\mathrm{limeal}}^k\|^2 - \alpha_k c_{\gamma,A}\|\lambda^{k+1}-\lambda^k\|^2.
\end{align*}
This finishes the proof of this lemma.
\end{proof}

\subsection{Proofs for Convergence of MEAL}

Based on the above lemmas, we give proofs of Theorem \ref{Theorem:Convergence-MEAL} and Proposition \ref{Proposition:globalconv-MEAL}.

\subsubsection{Proof of Theorem \ref{Theorem:Convergence-MEAL}}
\begin{proof}
We first establish the $o(1/\sqrt{k})$ rate of convergence under the \textit{implicit Lipschitz subgradient} assumption (Assumption \ref{Assump:MEAL}(b)) and then the convergence rate result under the \textit{implicit bounded subgradient} assumption (Assumption \ref{Assump:MEAL}(c)).

\textbf{(a)}
In the first case, $\beta_k = \beta$ and $\alpha_k = \alpha$.
Substituting \eqref{Eq:lambda-MEAL} into \eqref{Eq:1-step-progress-MEAL} yields
\begin{align*}
&{\cal P}_{\beta}(x^{k},z^k,\lambda^k) - {\cal P}_{\beta}(x^{k+1},z^{k+1},\lambda^{k+1})
\geq \frac{1}{2}\gamma \eta(1-\eta/2)\|\nabla \phi_{\beta}(z^k,\lambda^k)\|^2\\
&+\left(\frac{(1-\gamma\rho)}{2\gamma} - 2\alpha (1+\gamma L_f)^2  \right)\|x^{k+1}-x^k\|^2
+ \frac{1}{4\gamma}(\frac{2}{\eta}-1)\|z^{k+1}-z^k\|^2 \\
&- 2\alpha \|z^k-z^{k-1}\|^2.
\end{align*}
By the definition \eqref{Eq:Lyapunov-seq-MEAL-S1} of ${\cal E}_{\mathrm{meal}}^k$, the above inequality implies
\begin{align}
\label{Eq:descent-MEAL-S1*}
{\cal E}_{\mathrm{meal}}^k - {\cal E}_{\mathrm{meal}}^{k+1}
&\geq \frac{1}{2}\gamma \eta(1-\eta/2)\|\nabla \phi_{\beta}(z^k,\lambda^k)\|^2 + \left(\frac{1}{4\gamma}(\frac{2}{\eta}-1) -2\alpha \right)\|z^{k+1}-z^k\|^2 \nonumber\\
&+\left(\frac{1-\gamma\rho}{2\gamma} - 2\alpha (1+\gamma L_f)^2  \right)\|x^{k+1}-x^k\|^2 \\
&\geq \frac{1}{2}\gamma \eta(1-\eta/2)\|\nabla \phi_{\beta}(z^k,\lambda^k)\|^2, \nonumber
\end{align}
where the second inequality holds due to the condition on $\alpha$.
Thus, claim (a) follows from the above inequality, Lemma \ref{Lemma:sequence-fixed} with $\tilde{\epsilon}_k=0$ and the lower boundedness of $\{{\cal E}_{\mathrm{meal}}^k\}$.

\textbf{(b)}
Similarly,
substituting \eqref{Eq:lambda-MEAL-S2} into \eqref{Eq:1-step-progress-MEAL} and using the definition \eqref{Eq:Lyapunov-seq-MEAL-S2} of $\tilde{\cal E}_{\mathrm{meal}}^k$, we have
\begin{align*}
&\tilde{\cal E}_{\mathrm{meal}}^k - \tilde{\cal E}_{\mathrm{meal}}^{k+1} \geq \frac{1}{2}\gamma \eta(1-\eta/2)\|\nabla \phi_{\beta_k}(z^k,\lambda^k)\|^2 - 12 \alpha_k \gamma^2 \hat{L}_f^2\\
&+\left(\frac{1-\gamma\rho}{2\gamma} - 3\alpha_k  \right)\|x^{k+1}-x^k\|^2
+ \left(\frac{1}{4\gamma}(\frac{2}{\eta}-1) - 3\alpha_{k+1} \right)\|z^{k+1}-z^k\|^2.
\end{align*}
With $\alpha_k = \frac{\alpha^*}{K}$,
\begin{align*}
&\tilde{\cal E}_{\mathrm{meal}}^k - \tilde{\cal E}_{\mathrm{meal}}^{k+1}
\geq \frac{1}{2}\gamma (1-\eta/2)\eta\|\nabla \phi_{\beta_k}(z^k,\lambda^k)\|^2 - 12 \alpha_k \gamma^2 \hat{L}_f^2,
\end{align*}
which yields  claim (b) by Lemma \ref{Lemma:sequence-varying} with $\tilde{\epsilon}_k=0$ and the lower boundedness of $\{\tilde{\cal E}_{\mathrm{meal}}^k\}$.
\end{proof}

\subsubsection{Proof of Proposition \ref{Proposition:globalconv-MEAL}}

\begin{proof}
With Lemma \ref{Lemma:existing-global-converg}, we only need to check conditions $(P1)$-$(P3)$ hold for MEAL.

\textbf{(a) Establishing $(P1)$}: With $a:= \frac{\gamma \eta(2-\eta)}{4\beta}$, we have $\frac{1+a}{\beta c_{\gamma,A}}=\alpha$ for $\alpha$ in \eqref{Eq:alpha}.
Substituting \eqref{Eq:lambda-MEAL} into \eqref{Eq:primal-descent-MEAL} with fixed $\beta_k$ yields
\begin{align*}
&{\cal P}_{\beta}(x^k,z^k,\lambda^k) - {\cal P}_{\beta}(x^{k+1},z^{k+1},\lambda^{k+1}) \geq (\frac{1-\rho\gamma}{2\gamma}-2\alpha (\gamma L_f+1)^2)\|x^{k+1}-x^k\|^2\\
& + \frac{1}{2\gamma}(\frac{2}{\eta}-1)\|z^{k+1}-z^k\|^2 - 2\alpha \|z^k-z^{k-1}\|^2 + a\beta^{-1}\|\lambda^{k+1}-\lambda^k\|^2.
\end{align*}
For the definition \eqref{Eq:Lyapunov-fun-MEAL} of ${\cal P}_{\mathrm{meal}}$ and the assumption on $\alpha$, we deduce from the above inequality:
\begin{align}
\label{Eq:sufficient-descent-MEAL}
&{\cal P}_{\mathrm{meal}}(y^k) - {\cal P}_{\mathrm{meal}}(y^{k+1})\geq (\frac{1-\rho\gamma}{2\gamma}-2\alpha (\gamma L_f+1)^2)\|x^{k+1}-x^k\|^2 \nonumber\\
& + \left( \frac{1}{2\gamma}(\frac{2}{\eta}-1) - 3\alpha \right)\|z^{k+1}-z^k\|^2 + \alpha \|z^k-z^{k-1}\|^2 + a\beta^{-1}\|\lambda^{k+1}-\lambda^k\|^2 \nonumber\\
&\geq c_1\|y^{k+1}-y^k\|^2,
\end{align}
where $c_1:= \min\left\{\frac{1-\rho\gamma}{2\gamma}-2\alpha (\gamma L_f +1)^2, \alpha, a\beta^{-1}\right\}$ by $\frac{1}{2\gamma}(\frac{2}{\eta}-1) -3\alpha \geq \alpha$. This yields $(P1)$ for MEAL.

\textbf{(b) Establishing $(P2)$}:
Note that ${\cal P}_{\mathrm{meal}}(y) = f(x)+\langle \lambda,Ax-b\rangle + \frac{\beta}{2}\|Ax-b\|^2+\frac{1}{2\gamma}\|x-z\|^2 + 3\alpha \|z - \hat{z}\|^2$.
The optimality condition from the update of $x^{k+1}$ in \eqref{alg:MEAL-reformulation} is
\begin{align*}
0\in \partial f(x^{k+1}) + A^T\lambda^{k+1} + \gamma^{-1}(x^{k+1}-z^k),
\end{align*}
which implies $\gamma^{-1}(z^k - z^{k+1}) + A^T(\lambda^{k+1}-\lambda^{k}) \in \partial_x {\cal P}_{\mathrm{meal}}(y^{k+1}).$
From the update of $z^{k+1}$ in \eqref{alg:MEAL-reformulation}, $z^{k+1}-x^{k+1}=-(1-\eta)\eta^{-1}(z^{k+1}-z^k)$ and thus
\begin{align*}
\partial_z {\cal P}_{\mathrm{meal}}(y^{k+1})
= \gamma^{-1}(z^{k+1}-x^{k+1}) + 6\alpha (z^{k+1}-z^k)
=\left(6\alpha  - \frac{1-\eta}{\eta\gamma}\right)(z^{k+1}-z^k).
\end{align*}
The update of $\lambda^{k+1}$ in \eqref{alg:MEAL-reformulation} yields $\partial_\lambda {\cal P}_{\mathrm{meal}}(y^{k+1}) = Ax^{k+1}-b = \beta^{-1}(\lambda^{k+1}-\lambda^k).$
Moreover, it is easy to show $\partial_{\hat{z}} {\cal P}_{\mathrm{meal}}(y^{k+1}) = 6\alpha  (z^{k}-z^{k+1}).$
Thus, let
\[
v^{k+1}:=
\left(
\begin{array}{c}
\gamma^{-1}(z^k - z^{k+1}) + A^T(\lambda^{k+1}-\lambda^{k})\\
\left(6\alpha  - \frac{1-\eta}{\eta\gamma}\right)(z^{k+1}-z^k)\\
\beta^{-1}(\lambda^{k+1}-\lambda^k)\\
6\alpha  (z^{k}-z^{k+1})
\end{array}
\right),
\]
which obeys $v^{k+1} \in \partial {\cal P}_{\mathrm{meal}}(y^{k+1})$ and
\begin{align*}
\|v^{k+1}\|
&\leq \left(\gamma^{-1}+\left|6\alpha - \frac{1-\eta}{\eta\gamma} \right| + 6\alpha \right)\|z^{k+1}-z^k\| + \beta^{-1}\|\lambda^{k+1}-\lambda^k\| + \|A^T(\lambda^{k+1}-\lambda^k)\|\\
&\leq \left(\gamma^{-1}+\left|6\alpha - \frac{1-\eta}{\eta\gamma} \right| + 6\alpha \right)\|z^{k+1}-z^k\| + \beta^{-1}\|\lambda^{k+1}-\lambda^k\| \\
&+(L_f+\gamma^{-1})\|x^{k+1}-x^k\| + \gamma^{-1}\|\hat{z}^{k+1}-\hat{z}^k\|,
\end{align*}
where the second inequality is due to \eqref{Eq:A-lambda-MEAL}. This yields $(P2)$ for MEAL.

\textbf{(c) Establishing $(P3)$}: $(P3)$ follows from the boundedness assumption of $\{y^k\}$, and the convergence of $\{{\cal P}_{\mathrm{meal}}(y^k)\}$ is implied by $(P1)$.
This finishes the proof.
\end{proof}

\subsection{Proof for Convergence of iMEAL}
In this subsection, we present the proof of Theorem \ref{Theorem:Convergence-iMEAL} for iMEAL \eqref{alg:iMEAL}.
\begin{proof}[of Theorem \ref{Theorem:Convergence-iMEAL}]
We first show the $o(1/\sqrt{k})$ rate of convergence under Assumption \ref{Assump:MEAL}(b) and then the convergence rate result under Assumption \ref{Assump:MEAL}(c).

\textbf{(a)}
In this case, we use a fixed $\beta_k = \beta$ and thus $\alpha_k = \alpha$.
Substituting the inequality in Lemma \ref{Lemma:dual-control-primal-iMEAL}(a) into \eqref{Eq:1-step-progress-iMEAL} in Lemma \ref{Lemma:1-step-progress-iMEAL} yields
\begin{align}
\label{Eq:primal-descent-iMEAL-S1}
&{\cal P}_{\beta}(x^{k},z^k,\lambda^k) - {\cal P}_{\beta}(x^{k+1},z^{k+1},\lambda^{k+1})
\geq \frac{1}{2}\gamma \eta(1-\eta/2)\|\nabla \phi_{\beta}(z^k,\lambda^k)\|^2\\
&+\left(\frac{1-\gamma\rho}{2\gamma} - 3\alpha (1+\gamma L_f)^2  \right)\|x^{k+1}-x^k\|^2 +\langle s^k, x^k - x^{k+1} \rangle  - 3\alpha \gamma^2 (\epsilon_k+\epsilon_{k-1})^2 \nonumber\\
&+ \frac{1}{4\gamma}(\frac{2}{\eta}-1)\|z^{k+1}-z^k\|^2 - 3\alpha \|z^k-z^{k-1}\|^2. \nonumber
\end{align}
Let $\delta := 2\left(\frac{(1-\gamma\rho)}{2\gamma} - 3\alpha (1+\gamma L_f)^2  \right)$.
By the assumption $0<\alpha < \min\big\{\frac{1-\gamma \rho}{6\gamma(1+\gamma L_f)^2}$,   $\frac{1}{12\gamma}(\frac{2}{\eta}-1)\Big\}$, we have $\delta>0$ and further
\begin{align*}
\langle s^k, x^k - x^{k+1} \rangle \geq -\frac{\delta}{2}\|x^{k+1}-x^k\|^2 - \frac{1}{2\delta} \|s^k\|^2 \geq -\frac{\delta}{2}\|x^{k+1}-x^k\|^2 - \frac{1}{2\delta} (\epsilon_k + \epsilon_{k-1})^2.
\end{align*}
Substituting this into \eqref{Eq:primal-descent-iMEAL-S1} and
noting the definition \eqref{Eq:Lyapunov-seq-iMEAL-S1} of ${\cal E}_{\mathrm{imeal}}^k$,
we have
\begin{align*}
{\cal E}_{\mathrm{imeal}}^k - {\cal E}_{\mathrm{imeal}}^{k+1} \geq \frac{1}{2}\gamma \eta(1-\eta/2)\|\nabla \phi_{\beta}(z^k,\lambda^k)\|^2 - (3\alpha \gamma^2 + \frac{1}{2\delta})(\epsilon_k + \epsilon_{k-1})^2,
\end{align*}
which yields claim (a) by the assumption $\sum_{k=1}^{\infty} (\epsilon_k)^2 <+\infty$ and Lemma \ref{Lemma:sequence-fixed}.

\textbf{(b)} Then we establish claim (b) under Assumption \ref{Assump:MEAL}(c).
Substituting the inequality in Lemma \ref{Lemma:dual-control-primal-iMEAL}(b) into \eqref{Eq:1-step-progress-iMEAL} in Lemma \ref{Lemma:1-step-progress-iMEAL} yields
\begin{align}
\label{Eq:primal-descent-iMEAL-S2}
&{\cal P}_{\beta_k}(x^{k},z^k,\lambda^k) - {\cal P}_{\beta_{k+1}}(x^{k+1},z^{k+1},\lambda^{k+1})
\geq \frac{1}{2}\gamma \eta(1-\eta/2)\|\nabla \phi_{\beta_k}(z^k,\lambda^k)\|^2 \nonumber\\
&+\left(\frac{(1-\gamma\rho)}{2\gamma} - 4\alpha_k  \right)\|x^{k+1}-x^k\|^2 +\langle s^k, x^k - x^{k+1} \rangle  - 4\gamma^2 \alpha_k (\epsilon_k+\epsilon_{k-1})^2 \nonumber\\
&+ \frac{1}{4\gamma}(\frac{2}{\eta}-1)\|z^{k+1}-z^k\|^2 - 4\alpha_k \|z^k-z^{k-1}\|^2 - 16 \alpha_k \gamma^2 \hat{L}_f^2.
\end{align}
Let $\hat{\alpha}^*:= \min\left\{ \frac{1-\rho \gamma}{8\gamma}, \frac{1}{16\gamma}(\frac{2}{\eta}-1)\right\}$ and  $\tilde{\delta}:= 2\left(\frac{(1-\gamma\rho)}{2\gamma} - 4\hat{\alpha}^*  \right)>0$. We have
\begin{align*}
\langle s^k, x^k - x^{k+1} \rangle \geq -\frac{\tilde{\delta}}{2}\|x^{k+1}-x^k\|^2 - \frac{1}{2\tilde{\delta}} \|s^k\|^2 \geq -\frac{\tilde{\delta}}{2}\|x^{k+1}-x^k\|^2 - \frac{1}{2\tilde{\delta}} (\epsilon_k + \epsilon_{k-1})^2.
\end{align*}
Substituting this into \eqref{Eq:primal-descent-iMEAL-S2}, and by the definition \eqref{Eq:Lyapunov-seq-iMEAL-S2} of $\tilde{\cal E}_{\mathrm{imeal}}^k$
and setting of $\alpha_k$, we have
\begin{align*}
&\tilde{\cal E}_{\mathrm{imeal}}^k - \tilde{\cal E}_{\mathrm{imeal}}^{k+1} \\
&\geq \frac{1}{2}\gamma (1-\eta/2) \eta\|\nabla \phi_{\beta_k}(z^k,\lambda^k)\|^2 - (4\alpha_k \gamma^2 + \frac{1}{2\tilde{\delta}})(\epsilon_k + \epsilon_{k-1})^2 - 16 \alpha_k \gamma^2 \hat{L}_f^2,
\end{align*}
which yields claim (b) by the assumption $\sum_{k=1}^{\infty} (\epsilon_k)^2 <+\infty$ and Lemma \ref{Lemma:sequence-varying}.
\end{proof}

\subsection{Proofs for Convergence of LiMEAL}

Now, we show proofs of main convergence theorems for LiMEAL \eqref{alg:LiMEAL}.

\subsubsection{Proof of Theorem \ref{Theorem:Convergence-LiMEAL}}

\begin{proof}
We first establish claim (a) and then claim (b) under the associated assumptions.

\textbf{(a)} In this case, a fixed $\beta_k$ is used. Substituting \eqref{Eq:lambda-LiMEAL} into \eqref{Eq:1-step-progress-LiMEAL} yields
\begin{align*}
&{\cal P}_{\beta}(x^{k},z^k,\lambda^k) - {\cal P}_{\beta}(x^{k+1},z^{k+1},\lambda^{k+1})
\geq \frac{1}{4}\gamma (1-\eta/2)\eta \|g_{\mathrm{limeal}}^k\|^2\\
&+ \left(\frac{1-\gamma(\rho_g+L_h)}{2\gamma} - \frac{1}{4}\gamma (2-\eta)\eta L_h^2 - 3(1+\gamma L_g)^2 \alpha\right)\|x^{k+1}-x^k\|^2 \\
&+ \frac{1}{4\gamma}(\frac{2}{\eta}-1)\|z^{k+1}-z^k\|^2 - 3\alpha (\gamma^2L_h^2 \|x^k-x^{k-1}\|^2 + \|z^k - z^{k-1}\|^2).
\end{align*}
By the definition \eqref{Eq:Lyapunov-seq-LiMEAL-S1} of ${\cal E}^k_{\mathrm{limeal}}$, the above inequality implies
\begin{align}
\label{Eq:descent-LiMEAL-S1*}
&{\cal E}_{\mathrm{limeal}}^k - {\cal E}_{\mathrm{limeal}}^{k+1}
\geq \frac{1}{4}\gamma (1-\eta/2)\eta \|g_{\mathrm{limeal}}^k\|^2 + \left(\frac{1}{4\gamma}(\frac{2}{\eta}-1) -3\alpha \right)\|z^{k+1}-z^k\|^2\\
&+ \left(\frac{1-\gamma(\rho_g+L_h)}{2\gamma} - \frac{1}{4}\gamma (2-\eta)\eta L_h^2 - 3\alpha \left((1+\gamma L_g)^2 + \gamma^2L_h^2 \right)\right)\|x^{k+1}-x^k\|^2 \nonumber\\
& \geq  \frac{1}{4}\gamma (1-\eta/2)\eta \|g_{\mathrm{limeal}}^k\|^2, \nonumber
\end{align}
where the second inequality holds under the conditions in Theorem \ref{Theorem:Convergence-LiMEAL}(a).
This shows the claim (a) by Lemma \ref{Lemma:sequence-fixed} and the lower boundedness of $\{{\cal E}_{\mathrm{limeal}}^k\}$.

\textbf{(b)} Similarly,
substituting \eqref{Eq:lambda-LiMEAL-S2} into \eqref{Eq:1-step-progress-LiMEAL} and using the definitions of ${\alpha}_k$ in \eqref{Eq:alphak} and $\tilde{\cal E}^k_{\mathrm{limeal}}$ in \eqref{Eq:Lyapunov-seq-LiMEAL-S2}, we obtain
\begin{align*}
&\tilde{\cal E}_{\mathrm{limeal}}^k - \tilde{\cal E}_{\mathrm{limeal}}^{k+1}\\
&\geq \frac{1}{4}\gamma (1-\eta/2)\eta_k \|g_{\mathrm{limeal}}^k\|^2 - 16{\alpha}_k \gamma^2 \hat{L}_g^2 + \left(\frac{1}{4\gamma}(\frac{2}{\eta}-1) -4{\alpha}_{k+1}\right)\|z^{k+1}-z^k\|^2\\
&+\left(\frac{(1-\gamma(\rho_g+L_h))}{2\gamma} - \frac{1}{4}\gamma (2-\eta)\eta L_h^2- 4 \alpha_k - 4\gamma^2L_h^2 \alpha_{k+1} \right) \|x^{k+1}-x^k\|^2 \nonumber\\
&\geq \frac{1}{4}c\gamma \eta \|g_{\mathrm{limeal}}^k\|^2 - 16{\alpha}_k \gamma^2 \hat{L}_g^2,
\end{align*}
where the second inequality is due to the settings of parameters presented in Theorem \ref{Theorem:Convergence-LiMEAL}(b).
This inequality shows claim (b) by Lemma \ref{Lemma:sequence-varying} and the lower boundedness of $\{\tilde{\cal E}_{\mathrm{limeal}}^k\}$.
\end{proof}

\subsubsection{Proof of Proposition \ref{Proposition:globalconv-LiMEAL}}

\begin{proof}
By Lemma \ref{Lemma:existing-global-converg}, we only need to verify conditions $(P1)$-$(P3)$ hold for LiMEAL.

\textbf{(a) Establishing $(P1)$}: Similar to the proof of Theorem \ref{Theorem:Convergence-MEAL}, let $a:= \frac{\gamma \eta(2-\eta)}{4\beta}$. Then $\frac{1+a}{\beta c_{\gamma,A}}=\alpha$, where $\alpha$ is defined in \eqref{Eq:alpha}.
Substituting \eqref{Eq:lambda-LiMEAL} into \eqref{Eq:primal-descent-LiMEAL} with fixed $\beta_k$  yields
\begin{align*}
&{\cal P}_{\beta}(x^k,z^k,\lambda^k) - {\cal P}_{\beta}(x^{k+1},z^{k+1},\lambda^{k+1})\\
&\geq \left(\frac{1-\gamma(\rho_g+L_h)}{2\gamma} - 3\alpha(1+\gamma L_g)^2 \right)\|x^{k+1}-x^k\|^2 - 3\alpha \gamma^2L_h^2 \|x^k-x^{k-1}\|^2 \\
&+\frac{1}{2\gamma}(\frac{2}{\eta}-1)\|z^{k+1}-z^k\|^2 - 3\alpha \|z^k-z^{k-1}\|^2 + a\beta^{-1}\|\lambda^{k+1}-\lambda^k\|^2.
\end{align*}
By the definition \eqref{Eq:Lyapunov-fun-LiMEAL} of ${\cal P}_{\mathrm{limeal}}$, the above inequality implies
\begin{align*}
&{\cal P}_{\mathrm{limeal}}(y^k) - {\cal P}_{\mathrm{limeal}}(y^{k+1}) \\
&\geq \left(\frac{1-\gamma(\rho_g+L_h)}{2\gamma} - 4\alpha \left((1+\gamma L_g)^2 + \gamma^2 L_h^2 \right) \right) \|x^{k+1}-x^k\|^2\\
&+\left(\frac{1}{2\gamma}\left(\frac{2}{\eta} -1\right) - 4\alpha \right) \|z^{k+1}-z^k\|^2  + a\beta^{-1}\|\lambda^{k+1}-\lambda^k\|^2\\
&+\alpha \left(\gamma^2L_h^2 \|\hat{x}^{k+1}-\hat{x}^k\|^2 + \|\hat{z}^{k+1} - \hat{z}^k\|^2 \right),
\end{align*}
which, with the assumptions on the parameters, implies $(P1)$ for LiMEAL.

\textbf{(b) Establishing $(P2)$}:
Note that ${\cal P}_{\mathrm{limeal}}(y) = f(x)+\langle \lambda,Ax-b\rangle + \frac{\beta}{2}\|Ax-b\|^2+\frac{1}{2\gamma}\|x-z\|^2 + 4\alpha\gamma^2L_h^2 \|x-\hat{x}\|^2 + 4\alpha \|z - \hat{z}\|^2$.
The update of $x^{k+1}$ in \eqref{alg:LiMEAL} has the optimality condition
\begin{align*}
0\in \partial g(x^{k+1}) +\nabla h(x^k) + A^T\lambda^{k+1} + \gamma^{-1}(x^{k+1}-z^k),
\end{align*}
which implies
\begin{align*}
&(\nabla h(x^{k+1})-\nabla h(x^k)) + 8\gamma^2 L_h^2 \alpha (x^{k+1}-x^k)\\
&+\gamma^{-1}(z^k - z^{k+1}) + A^T(\lambda^{k+1}-\lambda^{k}) \in \partial_x {\cal P}_{\mathrm{limeal}}(y^{k+1}).
\end{align*}
The derivations for the other terms are straightforward and similar to those in the proof of Proposition \ref{Proposition:globalconv-MEAL}.
We directly show the final estimate: for some $v^{k+1} \in \partial {\cal P}_{\mathrm{limeal}}(y^{k+1})$,
\begin{align*}
\|v^{k+1}\|
&\leq \left(L_h+L_g+\gamma^{-1} + 16\alpha\gamma^2L_h^2\right)\|x^{k+1}-x^k\| \\
&+ \left(\gamma^{-1} + \left| 8\alpha - \frac{1-\eta}{\eta} \right| + 8\alpha \right) \|z^{k+1}-z^k\|\\
&+ \beta^{-1} \|\lambda^{k+1}-\lambda^k\| + L_h \|\hat{x}^{k+1} - \hat{x}^k\| + \gamma^{-1}\|\hat{z}^{k+1}-\hat{z}^k\|,
\end{align*}
which yields $(P2)$ for LiMEAL.

\textbf{(c) Establishing $(P3)$}: $(P3)$ follows from the boundedness assumption of $\{y^k\}$ and the convergence of $\{{\cal P}_{\mathrm{limeal}}(y^k)\}$ by $(P1)$.
This finishes the proof.
\end{proof}

\section{Discussions on Boundedness and Related Work}
\label{sc:discussion}

In this section, we discuss how to ensure the bounded sequences and then compare our results to related other work. 

\subsection{Discussions on Boundedness of Sequence}
\label{sc:discussion-boundedness}

Theorem \ref{Theorem:Convergence-MEAL} imposes the condition of lower boundedness of $\{{\cal E}_{\mathrm{meal}}^k\}$ and Proposition \ref{Proposition:globalconv-MEAL} does with boundedness of the generated sequence $\{(x^k,z^k,\lambda^k)\}$.
In this section, we provide some sufficient conditions to guarantee the former and then the latter boundedness conditions.

Besides the $\rho$-weak convexity of $f$ (implying the curvature of $f$ is lower bounded by $\rho$), we impose the coerciveness on the constrained problem \eqref{Eq:problem} as follows.

\begin{assumption}[Coercivity]
\label{Assumption:coercive}
The minimal value $f^*:= \inf_{x\in {\cal X}} f(x) $ is finite (recall ${\cal X}:=\{x: Ax=b\}$), and $f$ is coercive over the set ${\cal X}$,  that is, $f(x) \rightarrow \infty$ if $x\in {\cal X}$ and $\|x\| \rightarrow \infty$.
\end{assumption}


The coercive assumption is a common condition used to obtain the boundedness of the sequence, for example, used in~\cite[Assumption A1]{Wang19} for the nonconvex ADMM.
Particularly, let $(x^0,z^0,\lambda^0)$ be a finite initial guess of MEAL and
\begin{align}
\label{Eq:def-E0}
{\cal E}^0:={\cal E}_{\mathrm{meal}}^1 < +\infty.
\end{align}
By  Assumption \ref{Assumption:coercive}, if $x\in {\cal X}$ and $f(x) \leq {\cal E}^0$, then there exists a positive constant ${\cal B}_0$ (possibly depending on ${\cal E}^0$) such that
$
\|x\| \leq {\cal B}_0.
$
Define another positive constant as
\begin{align}
\label{Eq:boundedness-B1}
{\cal B}_1 := {\cal B}_0 + \sqrt{2\rho^{-1} \cdot \max\{0,{\cal E}^0 - f^*\}}.
\end{align}
Given a $\gamma \in (0, 1/\rho)$ and $z\in \mathbb{R}^n$ with $\|z\|\leq {\cal B}_1$ and $u\in \mathrm{Im}(A)$, we define
\begin{align}
\label{Eq:def-x-u-z}
x(u;z):= \argmin_{\{x:Ax=u\}} \left\{f(x)+\frac{1}{2\gamma}\|x-z\|^2\right\}.
\end{align}
Since $f$ is $\rho$-weakly convex by Assumption \ref{Assump:MEAL}(a), then for any $\gamma \in (0, 1/\rho)$, the function $f(x)+\frac{1}{2\gamma}\|x-z\|^2$ is strongly convex with respect to $x$, and thus the above $x(u;z)$ is well-defined and unique for any given $z \in \mathbb{R}^n$ and $u\in \mathrm{Im}(A)$.
Motivated by \cite[Ch 5.6.3]{Boyd04}, we impose some \textit{local stability} on $x(u;z)$ defined in \eqref{Eq:def-x-u-z}.

\begin{assumption}[Local stability]
\label{Assumption:Lipschitz-min-path}
For any given $z\in \mathbb{R}^n$ with $\|z\|\leq {\cal B}_1$, there exist a $\delta>0$ and a finite positive constant $\bar{M}$ (possibly depending on $A$, ${\cal B}_1$ and $\delta$) such that
\[
\|x(u;z)-x(b;z)\| \leq \bar{M} \|u-b\|, \ \forall u \in  \mathrm{Im}(A) \cap \{v: \|v-b\| \leq \delta\}.
\]
\end{assumption}


The above \textit{local stability} assumption is also related to the \textit{Lipschitz sub-minimization path} assumption suggested in \cite[Assumption A3]{Wang19}.
As discussed in \cite{Wang19}, the \textit{Lipschitz sub-minimization path} assumption relaxes the more stringent \textit{full-rank} assumption used in the literature (see the discussions in \cite[Sections 2.2 and 4.1]{Wang19} and references therein).
As $\{z\in \mathbb{R}^n: \|z\|\leq {\cal B}_1\}$ is a compact set, $\bar{M}$ can be taken as the supremum of these 
stability constants over this compact set.
Based on Assumption \ref{Assumption:Lipschitz-min-path}, we have the following lemma.


\begin{lemma}
\label{Lemma:A-control}
Let $\{(x^k,z^k,\lambda^k)\}$ be the sequence generated by MEAL \eqref{alg:MEAL-reformulation} with fixed $\beta>0$ and $\eta>0$. If $\gamma \in (0,1/\rho)$, $\|z^k\|\leq {\cal B}_1$ and $\|Ax^{k+1}-b\|\leq \delta$, there holds
\[
\|x^{k+1}-x(b;z^k)\| \leq  \bar{M}\|Ax^{k+1}-b\|, \ \forall k \in \mathbb{N}.
\]
\end{lemma}

\begin{proof}
Let $u^{k+1} = Ax^{k+1}$. By the update of $x^{k+1}$ in \eqref{alg:MEAL-reformulation}, there holds
\[
{\cal P}_{\beta}(x^{k+1},z^k,\lambda^{k}) \leq {\cal P}_{\beta}(x(u^{k+1};z^k),z^k,\lambda^{k}).
\]
Noting that $Ax(u^{k+1};z^k) = Ax^{k+1}$ due to its definition in \eqref{Eq:def-x-u-z}, the above inequality implies
\[
f(x^{k+1}) + \frac{1}{2\gamma}\|x^{k+1}-z^k\|^2 \leq f(x(u^{k+1};z^k)) + \frac{1}{2\gamma}\|x(u^{k+1};z^k)-z^k\|^2.
\]
By the definition of $x(u^{k+1};z^k)$ in \eqref{Eq:def-x-u-z} again and noting that $Ax^{k+1} = u^{k+1}$, we have
\[
f(x^{k+1}) + \frac{1}{2\gamma}\|x^{k+1}-z^k\|^2 \geq f(x(u^{k+1};z^k)) + \frac{1}{2\gamma}\|x(u^{k+1};z^k)-z^k\|^2.
\]
These two inequalities imply
\[
f(x^{k+1}) + \frac{1}{2\gamma}\|x^{k+1}-z^k\|^2 = f(x(u^{k+1};z^k)) + \frac{1}{2\gamma}\|x(u^{k+1};z^k)-z^k\|^2,
\]
which yields
\[
x^{k+1} = x(u^{k+1};z^k) = x(Ax^{k+1};z^k)
\]
by the strong convexity of function $f(x)+\frac{1}{2\gamma}\|x-z^k\|^2$ for any $\gamma \in (0,1/\rho)$ and thus the uniqueness of $x(u^{k+1};z^k)$.
Then by Assumption \ref{Assumption:Lipschitz-min-path}, we yield the desired result.
\end{proof}

Based on the above assumptions, we establish the lower boundedness of $\{{\cal E}_{\mathrm{meal}}^k\}$ and the boundedness of $\{(x^k, z^k, \lambda^k)\}$ as follows.



\begin{proposition}
\label{Propos:boundedness-x-z}
Let $\{(x^k, z^k, \lambda^k)\}_{k\in \mathbb{N}}$ be a sequence generated by MEAL \eqref{alg:MEAL-reformulation} with a finite initial guess $(x^0,z^0,\lambda^0)$ such that $\|z^0\|\leq {\cal B}_1$, where ${\cal B}_1$ is defined in \eqref{Eq:boundedness-B1}.
Suppose that Assumptions \ref{Assump:feasibleset}, \ref{Assump:MEAL}(a)-(b) and \ref{Assumption:coercive} hold and further Assumption \ref{Assumption:Lipschitz-min-path} holds with some $0<\bar{M}<\frac{2}{\sqrt{\sigma_{\min}(A^TA)}}$. If $\gamma \in ( 0,\rho^{-1})$, $\eta \in (0,2)$ and
$\beta > \max\left\{\frac{1+\sqrt{1+\eta(2-\eta)\gamma c_{\gamma,A}\alpha_{\max}}}{2c_{\gamma,A}\alpha_{\max}}, \frac{a_2+\sqrt{a_2^2+4a_1a_3}}{2a_1} \right\},$
where $\alpha_{\max} := \min\left\{\frac{1-\gamma \rho}{4\gamma(1+\gamma L_f)^2},  \frac{1}{8\gamma}(\frac{2}{\eta}-1)\right\}$, $c_{\gamma,A} = \gamma^2 \sigma_{\min}(A^TA)$, $a_1 = 4-\bar{M}^2\sigma_{\min}(A^TA)$, $a_2 = 4(\bar{L}+\gamma^{-1})\bar{M}^2 - \gamma \eta(2-\eta)$, $a_3 = (1+\gamma \bar{L})\eta(2-\eta)\bar{M}^2$ and $\bar{L}=\rho+2L_f$,
then  the following hold:
\begin{enumerate}
\item[(a)] $\{{\cal E}_{\mathrm{meal}}^k\}$ is lower bounded;

\item[(b)] $\{(x^k,z^k)\}$ is bounded; and

\item[(c)] if further $\lambda^0 \in \mathrm{Null}(A^T)$ (the null space of $A^T$) and $\|\nabla {\cal M}_{\gamma,f}(w^1)\|$ is finite with $w^1 = z^0 - \gamma A^T\lambda^1$, then $\{\lambda^k\}$ is bounded.
\end{enumerate}
\end{proposition}

\begin{proof}
In order to prove this proposition, we firstly establish the following claim for sufficiently large $k$:

\textbf{Claim A:}
\textit{If $\|z^{k-1}\|\leq {\cal B}_1, \|Ax^k-b\| \leq \delta, \forall k\geq k_0$ for some sufficiently large $k_0$, then  ${\cal E}_{\mathrm{meal}}^k \geq f^*$, and $\|z^{k}\|\leq {\cal B}_1$ and $\|x^k\| \leq {\cal B}_2$.}

By Theorem \ref{Theorem:Convergence-MEAL}(a), such $k_0$ does exist due to the lower boundedness of $\{{\cal E}_{meal}^k\}$ for all finite $k$ and thus $\xi_{\mathrm{meal}}^k \leq \hat{c}/\sqrt{k}$ for some constant $\hat{c}>0$ (implying $\|Ax^k-b\|$ is sufficiently small with a sufficiently large $k$ ).

In the next, we show \textbf{Claim A}.
By the definition \eqref{Eq:Lyapunov-seq-MEAL-S1}  of ${\cal E}_{\mathrm{meal}}^k$, we have
\begin{align*}
{\cal E}_{\mathrm{meal}}^k
&= f(x^k) + \langle \lambda^{k}, Ax^k-b \rangle + \frac{\beta}{2}\|Ax^k-b\|^2 + \frac{1}{2\gamma}\|x^k-z^k\|^2 + 2\alpha \|z^k-z^{k-1}\|^2\\
&=f(x^k) + \langle A^T\lambda^{k}, x^k-\bar{x}^k \rangle + \frac{\beta}{2}\|Ax^k-b\|^2 + \frac{1}{2\gamma}\|x^k-z^k\|^2 + 2\alpha \|z^k-z^{k-1}\|^2,
\end{align*}
where
\[\bar{x}^k:=x(b;z^{k-1})
\]
as defined in \eqref{Eq:def-x-u-z}. Let $\bar{\lambda}^k$ be the associated optimal Lagrangian multiplier of $\bar{x}^k$ and $\bar{w}^k = z^{k-1} - \gamma A^T \bar{\lambda}^k$. Then we have
\[
\bar{x}^k = \mathrm{Prox}_{\gamma,f}(\bar{w}^k),
\]
and $\nabla {\cal M}_{\gamma,f}(\bar{w}^k) \in \partial f(\bar{x}^k)$.
By \eqref{Eq:A-lambda} in the proof of Lemma \ref{Lemma:dual-control-primal-MEAL}, we have
\begin{align*}
A^T\lambda^{k} = -\nabla {\cal M}_{\gamma,f}(w^k) - \gamma^{-1}(x^k-z^{k-1}),
\end{align*}
and $\nabla {\cal M}_{\gamma,f}(w^k) \in \partial f(x^k)$, where $w^k = z^{k-1} - \gamma A^T \lambda^{k}$.
Substituting the above equation into the previous equality yields
\begin{align}
{\cal E}_{\mathrm{meal}}^k
&= f(x^k) + \langle \nabla {\cal M}_{\gamma,f}(w^k),\bar{x}^k - x^k \rangle + \frac{\beta}{2}\|Ax^k-b\|^2 \label{Eq:key-ineq1}\\
&+\gamma^{-1}\langle x^k-z^{k-1}, \bar{x}^k - x^k \rangle + \frac{1}{2\gamma}\|x^k-z^k\|^2 + 2 \alpha \|z^k - z^{k-1}\|^2. \nonumber
\end{align}
Noting that $\nabla {\cal M}_{\gamma,f}(\bar{w}^k) \in \partial f(\bar{x}^k)$ and by the $\rho$-weak convexity of $f$, we have
\[f(x^k) \geq f(\bar{x}^k) + \langle \nabla{\cal M}_{\gamma,f}(\bar{w}^k),x^k-\bar{x}^k\rangle - \frac{\rho}{2}\|x^k-\bar{x}^k\|^2,
\]
which implies
\begin{align*}
&f(x^k) + \langle \nabla{\cal M}_{\gamma,f}(w^k),\bar{x}^k-x^k\rangle \\
&\geq f(\bar{x}^k) - \frac{\rho}{2}\|\bar{x}^k - x^k\|^2 - \langle \nabla{\cal M}_{\gamma,f}(\bar{w}^k)-\nabla{\cal M}_{\gamma,f}(w^k),\bar{x}^k-x^k\rangle\\
&\geq f(\bar{x}^k) - \frac{\rho}{2}\|\bar{x}^k - x^k\|^2 - \|\nabla{\cal M}_{\gamma,f}(\bar{w}^k)-\nabla{\cal M}_{\gamma,f}(w^k)\|\cdot \|\bar{x}^k-x^k\|.
\end{align*}
By the implicit Lipschitz subgradient assumption (i.e., Assumption \ref{Assump:MEAL} (b)) and the definition of $\bar{L}:= \rho+2L_f$, the above inequality yields
\begin{align}
\label{Eq:key-ineq2}
f(x^k)+\langle \nabla {\cal M}_{\gamma,f}(w^k), \bar{x}^k - x^k \rangle \geq f(\bar{x}^k) - \frac{\bar{L}}{2}\|\bar{x}^k-x^k\|^2.
\end{align}
Moreover, it is easy to show that
\begin{align}
\label{Eq:key-ineq3}
&\gamma^{-1}\langle x^k - z^{k-1}, \bar{x}^k - x^k \rangle + \frac{1}{2\gamma}\|x^k-z^k\|^2 + 2 \alpha \|z^k - z^{k-1}\|^2  \\
& = \frac{1}{2\gamma}\|\bar{x}^k-z^k\|^2 - \frac{1}{2\gamma}\|\bar{x}^k-x^k\|^2 + \gamma^{-1} \langle z^k - z^{k-1}, \bar{x}^k - x^k\rangle + 2\alpha \|z^k - z^{k-1}\|^2 \nonumber\\
&=\frac{1}{2\gamma}\|\bar{x}^k-z^k\|^2 - \left(\frac{1}{2\gamma} + \frac{1}{8 \alpha \gamma^2}\right)\|\bar{x}^k-x^k\|^2
+2\alpha \left\|(z^k - z^{k-1}) + \frac{1}{4 \alpha \gamma} (\bar{x}^k - x^k)\right\|^2. \nonumber
\end{align}
Substituting \eqref{Eq:key-ineq2}-\eqref{Eq:key-ineq3} into \eqref{Eq:key-ineq1} and by Lemma \ref{Lemma:A-control}, we have
\begin{align}
{\cal E}_{\mathrm{meal}}^k
&\geq f(\bar{x}^k) + \frac{1}{2\gamma}\|\bar{x}^k - z^k\|^2
+2\alpha \left\|(z^k - z^{k-1}) + \frac{1}{4\alpha \gamma} (\bar{x}^k - x^k)\right\|^2 \nonumber\\
&+ \frac{1}{2} \left[\beta - \left(\frac{1}{4\alpha\gamma^2}+\bar{L}+\gamma^{-1} \right)\bar{M}^2\right] \|Ax^k-b\|^2 \nonumber\\
&\geq f(\bar{x}^k) + \frac{1}{2\gamma}\|\bar{x}^k - z^k\|^2
+2\alpha \left\|(z^k - z^{k-1}) + \frac{1}{4\alpha\gamma} (\bar{x}^k - x^k)\right\|^2 \label{Eq:key-ineq4}\\
&\geq f^* + \frac{1}{2\gamma}\|\bar{x}^k - z^k\|^2
+2\alpha \left\|(z^k - z^{k-1}) + \frac{1}{4 \alpha \gamma} (\bar{x}^k - x^k)\right\|^2 \label{Eq:key-ineq5}\\
&> -\infty, \label{Eq:key-ineq6}
\end{align}
where the second inequality follows from the definition of $\alpha = \frac{2\beta+\gamma \eta(1-\eta/2)}{2\gamma^2\sigma_{\min}(A^TA)\beta^2}$ and the condition on $\beta$, the third inequality holds for $\bar{x}^k:= x(b;z^{k-1})$ and thus $A\bar{x}^k = b$ and $f(\bar{x}^k) \geq f^*$, and the final inequality is due to Assumption \ref{Assumption:coercive}.
The above inequality yields the lower boundedness of $\{{\cal E}_{\mathrm{meal}}^k\}$ in \textbf{Claim A}. Thus, clam (a) in this proposition holds.

Then, we show the boundedness of $\{(x^k,z^k)\}$ in \textbf{Claim A}.
By \eqref{Eq:key-ineq4} and \eqref{Eq:descent-MEAL-S1*}, we have
\begin{align*}
f(\bar{x}^k) \leq {\cal E}^0 := {\cal E}_{\mathrm{meal}}^1,
\end{align*}
which implies $\|\bar{x}^k\| \leq {\cal B}_0$
by Assumption \ref{Assumption:coercive}. By \eqref{Eq:key-ineq5} and the condition on $\gamma \in (0, \rho^{-1})$, we have $f^* + \frac{\rho}{2}\|\bar{x}^k - z^k\|^2 \leq f^* + \frac{1}{2\gamma}\|\bar{x}^k - z^k\|^2 \leq {\cal E}^0,$
which implies
\begin{align*}
\|z^k\| \leq {\cal B}_0 + \sqrt{2({\cal E}^0-f^*)/\rho} = {\cal B}_1.
\end{align*}
By \eqref{Eq:key-ineq5} again, we have $\left\|(z^k - z^{k-1}) + \frac{1}{4 \alpha \gamma} (\bar{x}^k - x^k)\right\|^2 \leq \frac{{\cal E}^0 - f^*}{2 \alpha},$
which, together with these existing bounds $\|z^{k-1}\|\leq {\cal B}_1$, $\|z^{k}\|\leq {\cal B}_1$ and $\|\bar{x}^k\| \leq {\cal B}_0$, yields
\begin{align}
\label{Eq:boundedness-B2}
\|x^k\| \leq {\cal B}_0 + 4\alpha\gamma \left(2{\cal B}_1 + \sqrt{\frac{{\cal E}^0 - f^*}{2\alpha}} \right) =: {\cal B}_2.
\end{align}
Thus, we have shown \textbf{Claim A}. Recursively, we can show that $\{x^k\}$ and $\{z^k\}$ are respectively bounded by ${\cal B}_2$ and ${\cal B}_1$ for any $k\geq 1$, that is, claim (b) in this proposition holds.

In the following, we show claim (c) of this proposition.
By the update of $\lambda^{k+1}$ in \eqref{alg:MEAL-reformulation}, it is easy to show $\lambda^k = \lambda^0 + \hat{\lambda}^k$, where $\hat{\lambda}^k = \beta \sum_{t=1}^{k} (Ax^t-b) \in \mathrm{Im}(A)$ by Assumption \ref{Assump:feasibleset}. Furthermore, by the assumption that $\lambda^0 \in \mathrm{Null}(A^T)$, we have
\begin{align}
\label{Eq:lambda-orth}
\langle \lambda^0, \hat{\lambda}^k \rangle =0, \ \forall k \geq 1.
\end{align}
By \eqref{Eq:A-lambda}, for any $k\geq 1$, we have
\begin{align*}
A^T\lambda^{k}= -(\nabla {\cal M}_{\gamma,f}(w^{k})-\nabla {\cal M}_{\gamma,f}(w^{1}))-\nabla {\cal M}_{\gamma,f}(w^{1})-\gamma^{-1}(x^{k}-z^{k-1}),
\end{align*}
where $w^k = z^{k-1} - \gamma A^T\lambda^k$. By Assumption \ref{Assump:MEAL}(b) and the boundedness of $\{(x^k,z^k)\}$ shown before, the above equation implies
\begin{align*}
\|A^T {\lambda}^k\|
&\leq L_f \|x^k - x^1\| + \|\nabla {\cal M}_{\gamma,f}(w^{1})\| + \gamma^{-1}\|x^{k}-z^{k-1}\| \\
&\leq \gamma^{-1}{\cal B}_1 + (2L_f+\gamma^{-1}){\cal B}_2 + \|\nabla {\cal M}_{\gamma,f}(w^{1})\|  <+\infty.
\end{align*}
By the relation $\lambda^k = \lambda^0 + \hat{\lambda}^k$ and \eqref{Eq:lambda-orth}, the above inequality implies
\begin{align*}
\|A^T\hat{\lambda}^k\| \leq \gamma^{-1}{\cal B}_1 + (2L_f+\gamma^{-1}){\cal B}_2 + \|\nabla {\cal M}_{\gamma,f}(w^{1})\|.
\end{align*}
Since $\hat{\lambda}^k \in \mathrm{Im}(A)$, the above inequality implies
\begin{align*}
\|\hat{\lambda}^k\| \leq \tilde{\sigma}_{\min}^{-1/2}(A^TA) \|A^T\hat{\lambda}^k\| \leq \tilde{\sigma}_{\min}^{-1/2}(A^TA) \left[\gamma^{-1}{\cal B}_1 + (2L_f+\gamma^{-1}){\cal B}_2 + \|\nabla {\cal M}_{\gamma,f}(w^{1})\|\right],
\end{align*}
which yields the boundedness of $\{\lambda^k\}$ by the triangle inequality.
This finishes the proof.
\end{proof}

The proof idea of claim (c) of this proposition is motivated by the proof of \cite[Lemma 3.1]{Zhang-Luo18}.
Based on Proposition \ref{Propos:boundedness-x-z},
we show the lower boundedness of the Lypunov function sequence and the boundedness of the sequence generated by MEAL.
Following the similar analysis of this section, we can obtain the similar boundeness results for both iMEAL and LiMEAL.

\subsection{Discussions on Related Work}
\label{sc:related-work}

When compared to the tightly related work \cite{Hajinezhad-Hong19,Hong17-Prox-PDA,Jiang19,Rockafellar76-PALM,Xie-Wright19,Zhang-Luo20,Zhang-Luo18}, this paper provides some slightly stronger convergence results under weaker conditions. The detailed discussions and comparisons with these works are shown as follows and presented in Tables \ref{Tab:comp-alg1} and \ref{Tab:comp-alg2}.

\begin{table}
\caption{Convergence results of our and related algorithms for problem \eqref{Eq:problem}}
\footnotesize
\begin{center}
\begin{tabular}{c|c|c|c|c}\hline
Algorithm & MEAL (our)  & iMEAL (our) & Prox-PDA \cite{Hong17-Prox-PDA} & Prox-ALM \cite{Xie-Wright19} \\\hline\hline
Assumption &\multicolumn{2}{|c|}{$f$: weakly convex,  \textit{imp-Lip} or \textit{imp-bound}}
& \multicolumn{2}{|c|}{$\nabla f$: Lipschitz}
\\\hline
{Iteration} & \textit{imp-Lip}: $o(\varepsilon^{-2})$ & \textit{imp-Lip}: $o(\varepsilon^{-2})$ & \multirow{2}*{${\cal O}(\varepsilon^{-2})$} & \multirow{2}*{${\cal O}(\varepsilon^{-2})$} \\
 complexity & \textit{imp-bound}: ${\cal O}(\varepsilon^{-2})$ & \textit{imp-bound}: ${\cal O}(\varepsilon^{-2})$ & ~ &~ \\\hline
{Global}  & \multirow{2}*{$\checkmark$ under K{\L}} & \multirow{2}*{--}  & \multirow{2}*{--} & \multirow{2}*{--} \\
Convergence & ~ & ~ & ~ &~ \\\hline
\end{tabular}
\end{center}
$\bullet$~\textit{imp-Lip}: the \textit{implicit Lipschitz subgradient} assumption \ref{Assump:MEAL}(b);\\
$\bullet$~\textit{imp-bound}: the \textit{implicit bounded subgradient} assumption \ref{Assump:MEAL}(c);\\
$\bullet$~\cite{Xie-Wright19} considers a nonlinear equality constraints $c(x)=0$ where $\nabla c$ is Lipschitz and bounded.
\label{Tab:comp-alg1}
\end{table}

\begin{table}
\caption{Convergence results of our and related algorithms for the composite optimization problem \eqref{Eq:problem-CP}.}
\footnotesize
\begin{center}
\begin{tabular}{c|c|c|c|c}\hline
Algorithm & LiMEAL (our)  & PProx-PDA \cite{Hajinezhad-Hong19} & Prox-iALM \cite{Zhang-Luo18} & S-prox-ALM \cite{Zhang-Luo20} \\\hline\hline
\multirow{3}*{Assumption} &$\nabla h$: Lipschitz,
& $\nabla h$: Lipschitz, & $\nabla h$: Lipschitz, & $\nabla h$: Lipschitz,
\\
~ & $g$: weakly convex, & $g$: convex, & $g:\iota_{\cal C}(x)$, & $g:\iota_{\cal P}(x)$, \\
~ & \textit{imp-Lip} or \textit{imp-bound} &$\partial g$: bounded &${\cal C}$: box constraint &${\cal P}$: polyhedral set \\\hline
{Iteration} & \textit{imp-Lip}: $o(\varepsilon^{-2})$ & \multirow{2}*{${\cal O}(\varepsilon^{-2})$} & \multirow{2}*{${\cal O}(\varepsilon^{-2})$} & \multirow{2}*{${\cal O}(\varepsilon^{-2})$} \\
 complexity & \textit{imp-bound}: ${\cal O}(\varepsilon^{-2})$ & ~ & ~ &~ \\\hline
{Global}  & \multirow{2}*{$\checkmark$ under K{\L}} & \multirow{2}*{--}  & $\checkmark$ for quadratic & \multirow{2}*{--} \\
Convergence & ~ & ~ & programming &~ \\\hline
\end{tabular}
\end{center}
\label{Tab:comp-alg2}
\end{table}

When reduced to the case of linear constraints, the proximal ALM suggested in \cite{Rockafellar76-PALM} is a special case of MEAL with $\eta = 1$, and the Lipschitz continuity of certain fundamental mapping at the origin \cite[p. 100]{Rockafellar76-PALM} generally implies the K{\L} property of the proximal augmented Lagrangian with exponent $1/2$ at some stationary point, and thus, the linear convergence of proximal ALM can be directly yielded by Proposition \ref{Proposition:globalconv-MEAL}(b).
Moreover, the proposed algorithms still work (in terms of convergence) for some constrained problems with nonconvex objectives and a fixed penalty parameter.

In \cite{Hong17-Prox-PDA}, a proximal primal-dual algorithm (named \textit{Prox-PDA}) was proposed for the linearly constrained problem \eqref{Eq:problem} with $b=0$.
Prox-PDA is shown as follows:
\begin{align*}
\text{(Prox-PDA)} \
\left\{
\begin{array}{l}
x^{k+1} = \argmin_{x\in \mathbb{R}^n} \ \left\{ f(x)+\langle \lambda^k,Ax\rangle + \frac{\beta}{2}\|Ax\|^2 + \frac{\beta}{2}\|x-x^k\|^2_{B^TB}\right\},\\
\lambda^{k+1} = \lambda^k + \beta Ax^{k+1},
\end{array}
\right.
\end{align*}
where $B$ is chosen such that $A^TA + B^TB \succeq \mathrm{I}_n$ (the identity matrix of size $n$).
To achieve a $\sqrt{\varepsilon}$-accurate stationary point,
the iteration complexity of Prox-PDA is ${\cal O}(\varepsilon^{-1})$ under the Lipschitz differentiability of $f$ (that is, $f$ is differentiable and has Lipschitz gradient) and the assumption that there exists some $\underline{f}>-\infty$ and some $\delta>0$ such that $f(x)+\frac{\delta}{2}\|Ax\|^2 \geq \underline{f}$ for any $x\in \mathbb{R}^n$. Such iteration complexity of Prox-PDA is consistent with the order of ${\cal O}(\varepsilon^{-2})$ to achieve an $\varepsilon$-accurate stationary point.
On one hand if we take $B = \mathrm{I}_n$ in Prox-PDA, then it reduces to MEAL with $\gamma = \beta^{-1}$ and $\eta=1$.
On the other hand, by our main Theorem \ref{Theorem:Convergence-MEAL}(a), the iteration complexity of the order of $o(\varepsilon^{-2})$ is slightly better than that of Prox-PDA, under weaker conditions (see, Assumption \ref{Assump:MEAL}(a)-(b)).
Moreover, we established the global convergence and rate of MEAL under the  K{\L} inequality, while such global convergence result is missing (though obtainable) for Prox-PDA in \cite{Hong17-Prox-PDA}.

A prox-linear variant of Prox-PDA (there dubbed \textit{PProx-PDA}) was proposed in the recent paper \cite{Hajinezhad-Hong19} for the linearly constrained problem \eqref{Eq:problem-CP} with a composite objective. Besides Lipschitz differentiability of $h$, the nonsmooth function $g$ is assumed to be convex with bounded subgradients. These assumptions used in \cite{Hajinezhad-Hong19} are stronger than ours in Assumption \ref{Assump:LiMEAL}(a), (b) and (d), while the yielded iteration complexity of LiMEAL (Theorem \ref{Theorem:Convergence-LiMEAL}(b)) is consistent with that of PProx-PDA in \cite[Theorem 1]{Hajinezhad-Hong19}.
Moreover, we establish the global convergence and rate of LiMEAL (Proposition \ref{Proposition:globalconv-LiMEAL}), which is missing (though obtainable) for PProx-PDA.

In \cite{Xie-Wright19}, an ${\cal O}(\varepsilon^{-2})$-iteration complexity of proximal ALM was established for the constrained problem with nonlinear equality constraints, under assumptions that the objective is differentiable and its gradient is both Lipschitz continuous and bounded, and that the Jacobian of the constraints is also Lipschitz continuous and bounded and satisfies a \textit{full-rank} property (see \cite[Assumption 1]{Xie-Wright19}). If we reduce their setting to linear constraints, their iteration complexity is slightly worse than ours and their assumptions are stronger (of course, except for the part on nonlinear constraints).

In \cite{Zhang-Luo18}, a very related algorithm (called \textit{Proximal Inexact Augmented Lagrangian Multiplier method}, dubbed Prox-iALM) was introduced for the following linearly constrained problem
\begin{align*}
\min_{x\in \mathbb{R}^n} \ h(x) \quad \mathrm{subject \ to} \quad Ax=b, \ x\in {\cal C},
\end{align*}
where ${\cal C}$ is a box constraint set. Subsequence convergence to a stationary point was established under
the following assumptions:
(a) the origin is in the relative interior of the set $\{Ax-b: x\in {\cal C}\}$;
(b) the strict complementarity condition \cite{Nocedal99} holds for the above constrained problem;
(c) $h$ is differentiable and has Lipschitz continuous gradient.
Moreover, the global convergence and linear rate of this algorithm was established for the quadratic programming, in which case, the augmented Lagrangian satisfies the K{\L} inequality with exponent $1/2$, by noticing
the connection between Luo-Tseng error bound and K{\L} inequality \cite{Li-Pong-KLexponent18}.
According to Theorem \ref{Theorem:Convergence-LiMEAL} and Proposition \ref{Proposition:globalconv-LiMEAL}, the established convergence results in this paper are more general and stronger than that in \cite{Zhang-Luo18} but under weaker assumptions. Particularly, besides the weaker assumption on $h$, the strict complementarity condition (b) is also removed in this paper for LiMEAL.

The algorithm studied in \cite{Zhang-Luo18} has been recently generalized to handle the linearly constrained problem with the polyhedral set in \cite{Zhang-Luo20} (dubbed \textit{S-prox-ALM}). Under the Lipschitz differentiability of the objective, the iteration complexity of the order ${\cal O}(\varepsilon^{-2})$ was established in \cite{Zhang-Luo20} for the S-prox-ALM algorithm. Such iteration complexity is consistent with LiMEAL as shown in Theorem \ref{Theorem:Convergence-LiMEAL}.
Besides these major differences between this paper and \cite{Zhang-Luo20,Zhang-Luo18}, the step sizes $\eta$ are more flexible for both MEAL and LiMEAL (only requiring $\eta \in (0,2)$), while the step sizes used in the algorithms in \cite{Zhang-Luo20,Zhang-Luo18} should be sufficiently small to guarantee the convergence.
Meanwhile, the Lyapunov function used in this paper is motivated by the Moreau envelope of the augmented Lagrangian, which is very different from the Lyapunov function used in \cite{Zhang-Luo20,Zhang-Luo18}.
Based on the defined Lyapunov function, our analysis is much simpler than that in \cite{Zhang-Luo20,Zhang-Luo18}.

\section{Numerical Experiments}
\label{sc:experiment}
We use two experiments to demonstrate the effectiveness of the proposed algorithms:
\begin{enumerate}
    \item The first experiment is based on a nonconvex quadratic program on which ALM with any bounded penalty parameter diverges~\cite[Proposition 1]{Wang19} but LiMEAL converges.
    \item The second experiment borrows a general quadratic program from~\cite[Sec. 6.2]{Zhang-Luo18} and LiMEAL outperforms \textit{Prox-iALM} suggested in \cite{Zhang-Luo18}.
\end{enumerate}
The source codes can be accessed at \url{https://github.com/JinshanZeng/MEAL}.

\subsection{ALM vs LiMEAL}
\label{sc:exp1}
Consider the following optimization problem from \cite[Proposition 1]{Wang19}:
\begin{align}\label{Eq:exp1}
    \min_{x,y\in \mathbb{R}} \ x^2-y^2,
    \quad \text{subject to} \  x=y, \ x\in [-1,1].
\end{align}
ALM with any bounded penalty parameter $\beta$ diverges on this problem. By Theorem \ref{Theorem:Convergence-LiMEAL} and Proposition \ref{Proposition:globalconv-LiMEAL}, LiMEAL converges exponentially fast since its augmented Lagrangian is a K{\L} function with an exponent of $1/2$. For both ALM and LiMEAL, we set the penalty parameter $\beta$ to 50. We set LiMEAL's proximal parameter $\gamma$ to $1/2$ and test three different values $\eta's$: $0.5, 1, 1.5$. The curves of objective $f(x^k,y^k)=(x^k)^2 - (y^k)^2$, constraint violation error $|x^k-y^k|$, multiplier sequences $\{\lambda^k\}$, and the norm of gradient of Moreau envelope in \eqref{Eq:stationary-LiMEAL}, which is the stationarity measure, are depicted in Fig. \ref{Fig:Exp1}.

Observe that ALM diverges: its multiplier sequence $\{\lambda^k\}$ oscillates between two distinct values (Fig. \ref{Fig:Exp1} (a)) and the constraint violation converges to a positive value (Fig. \ref{Fig:Exp1} (b)). Also observe that LiMEAL converges exponentially fast (Fig. \ref{Fig:Exp1} (c)--(e)) and achieves the optimal objective value of 0 in about 10 iterations (Fig. \ref{Fig:Exp1} (f)) with all $\eta$ values. This verifies Proposition \ref{Proposition:globalconv-LiMEAL}.

\begin{figure}[!t]
\begin{minipage}[b]{0.49\linewidth}
\centering
\includegraphics*[scale=.43]{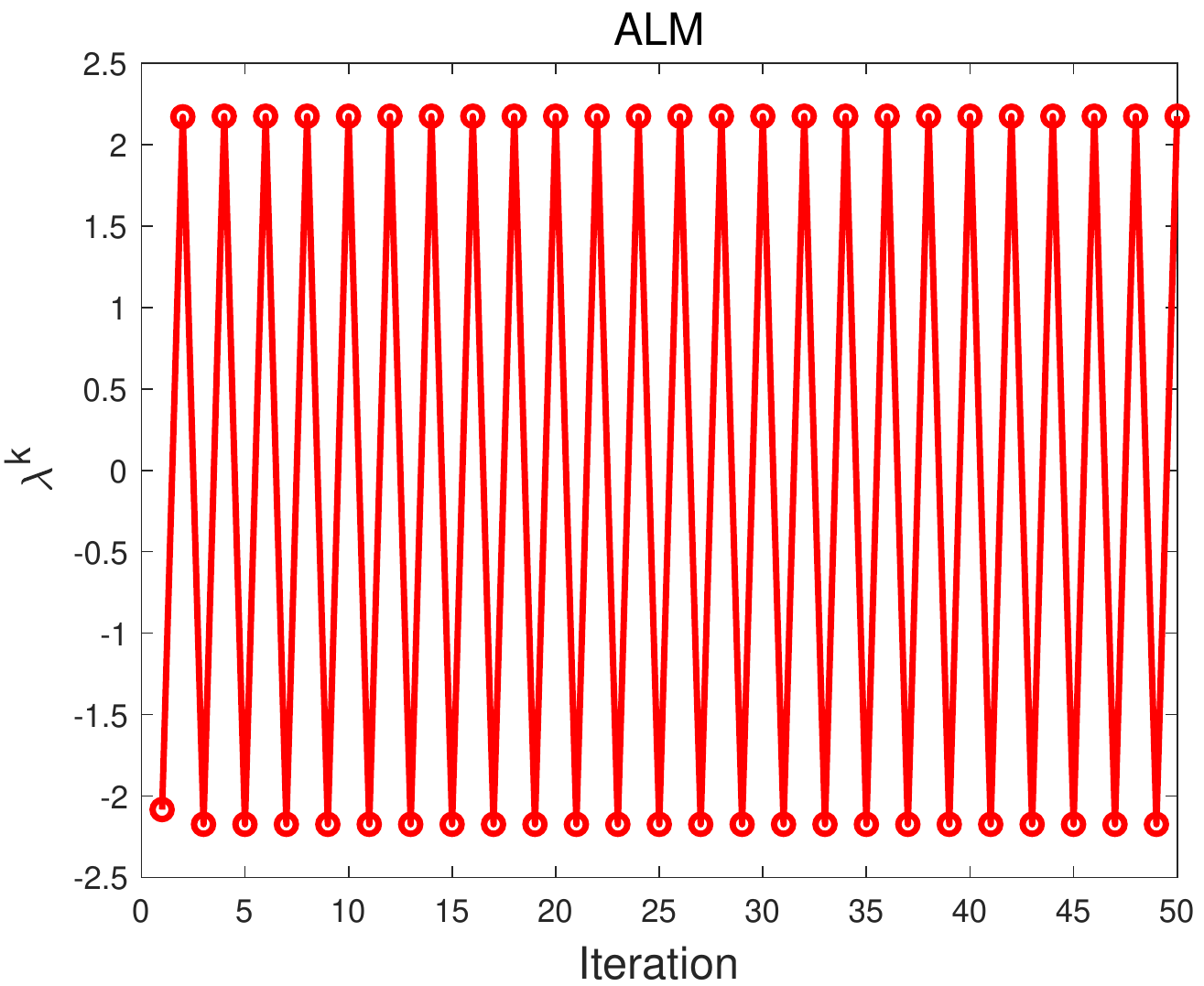}
\centerline{{\small (a) Divergent $\{\lambda^k\}$ of ALM}}
\end{minipage}
\hfill
\begin{minipage}[b]{0.49\linewidth}
\centering
\includegraphics*[scale=.43]{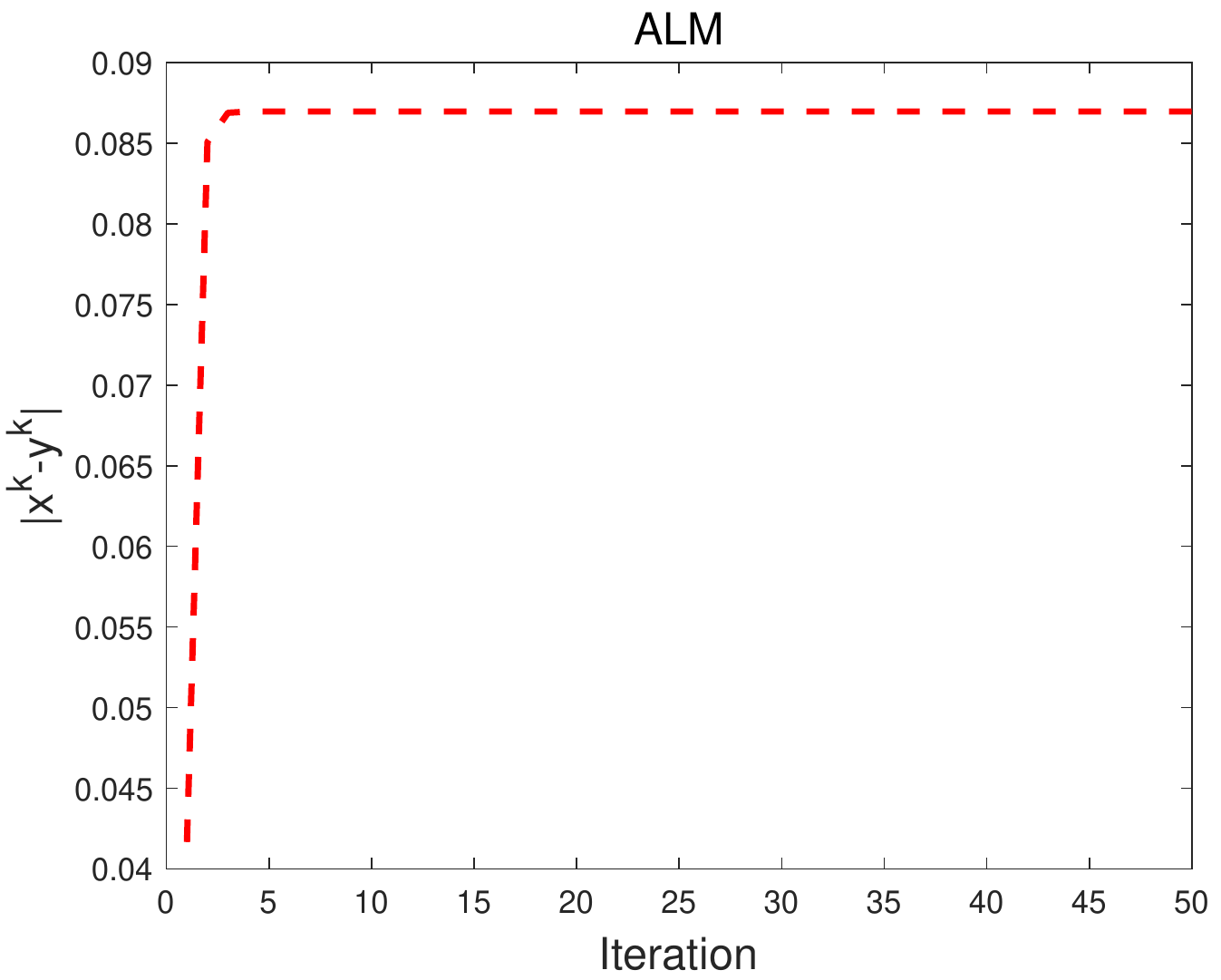}
\centerline{{\small (b) Constraint violation of ALM}}
\end{minipage}
\hfill
\begin{minipage}[b]{0.49\linewidth}
\centering
\includegraphics*[scale=.43]{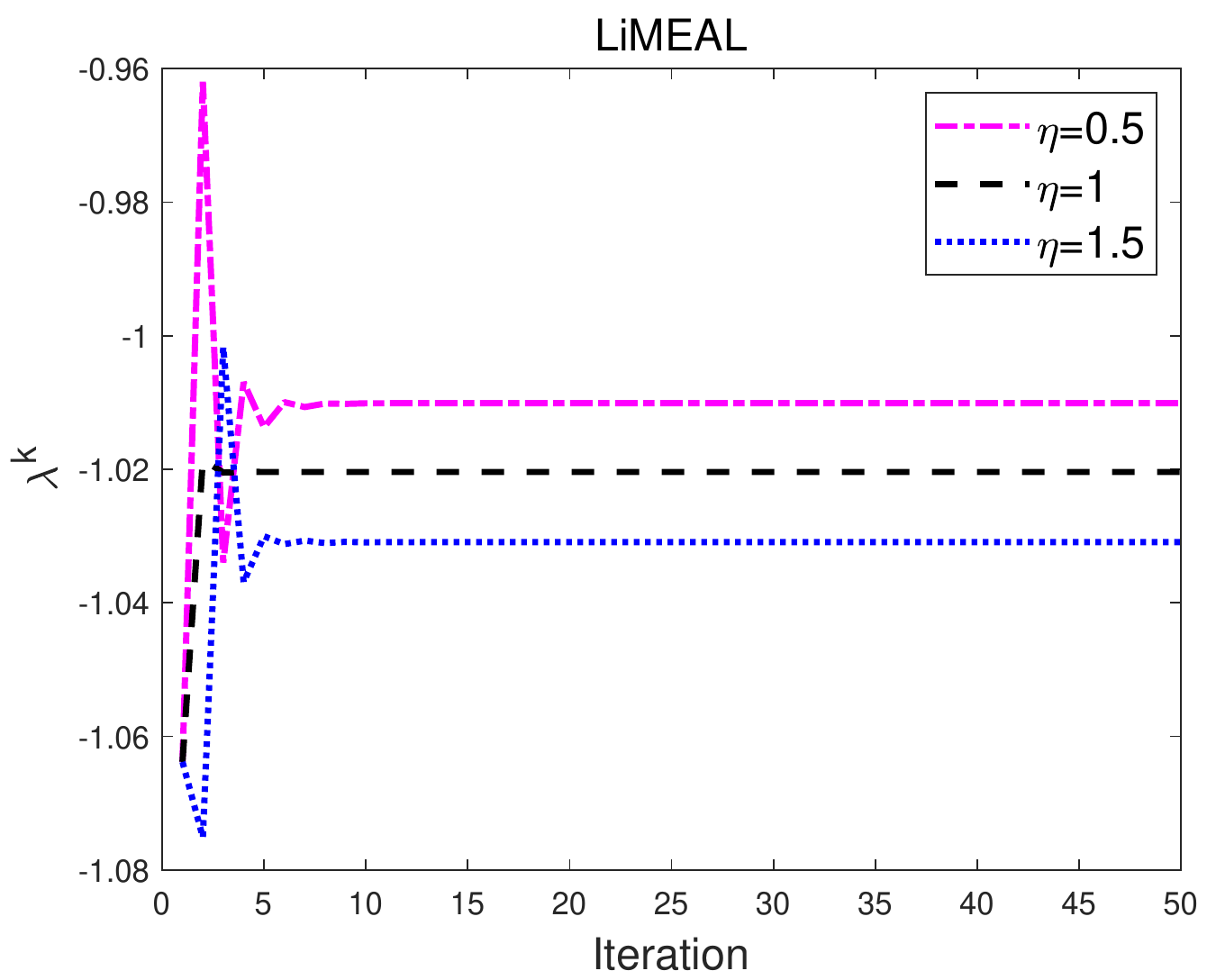}
\centerline{{\small (c) Convergent $\{\lambda^k\}$ of LiMEAL}}
\end{minipage}
\hfill
\begin{minipage}[b]{0.49\linewidth}
\centering
\includegraphics*[scale=.43]{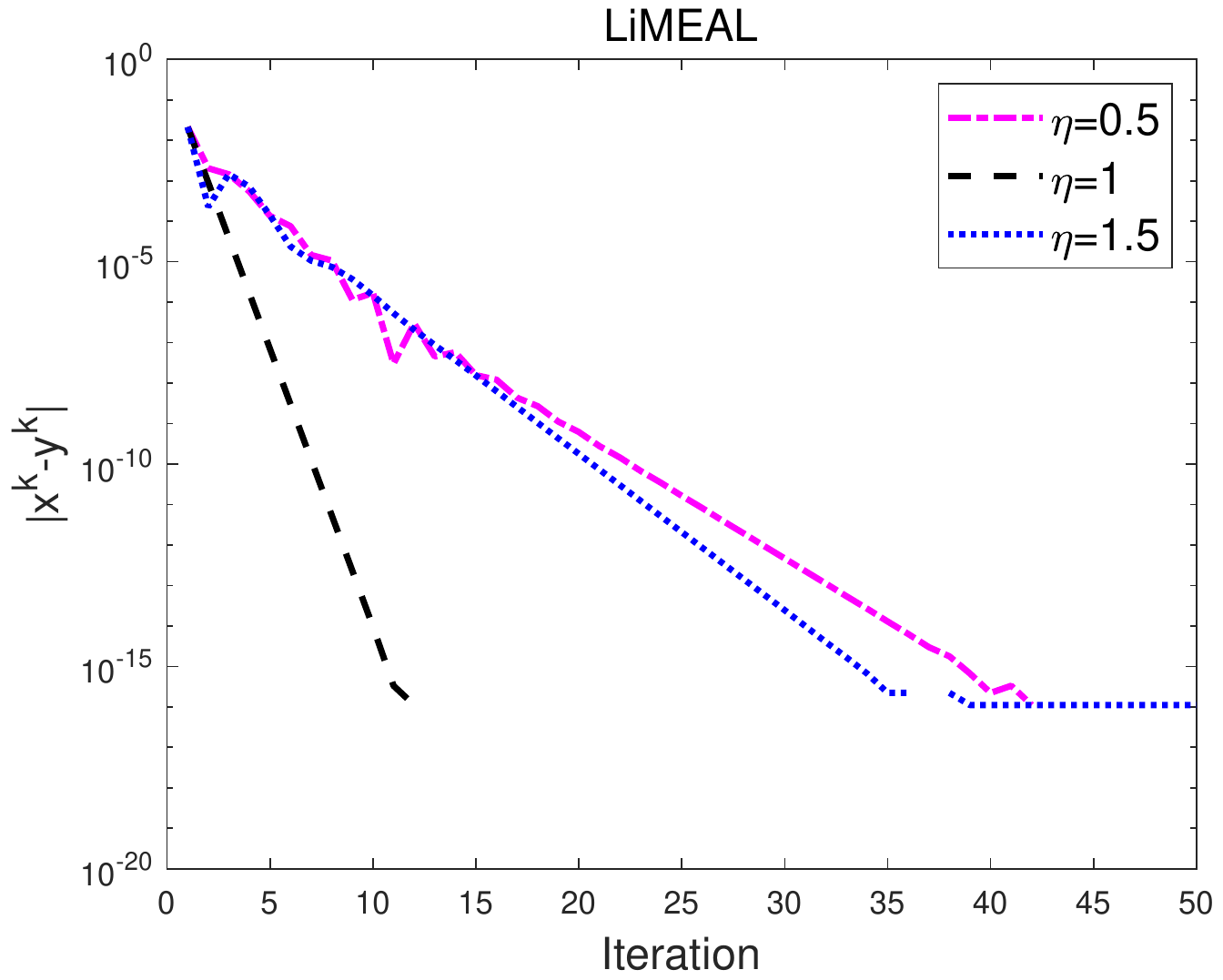}
\centerline{{\small (d) Constraint violation of LiMEAL}}
\end{minipage}
\hfill
\begin{minipage}[b]{0.49\linewidth}
\centering
\includegraphics*[scale=.43]{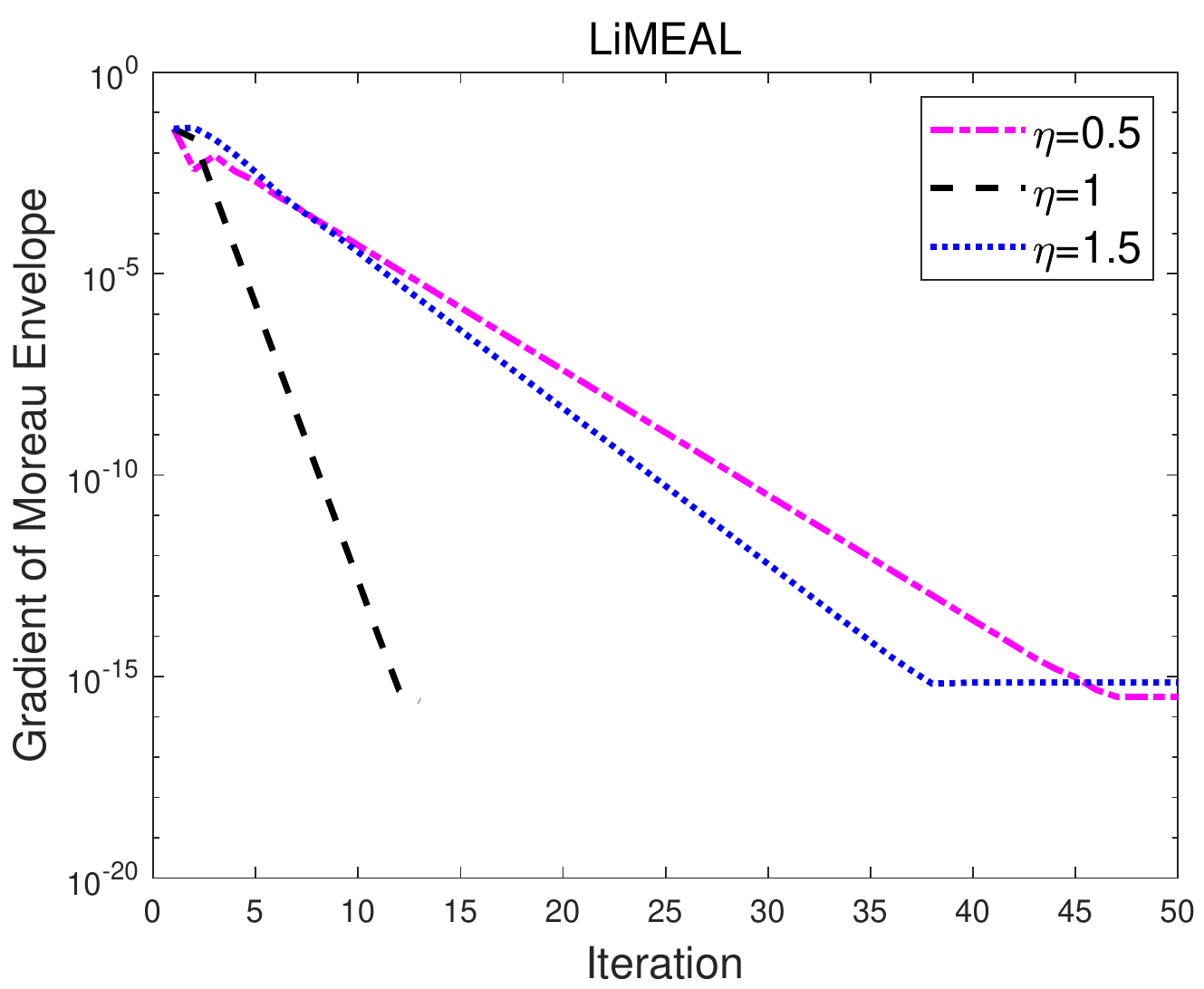}
\centerline{{\small (e) convergence rate of LiMEAL}}
\end{minipage}
\hfill
\begin{minipage}[b]{0.49\linewidth}
\centering
\includegraphics*[scale=.43]{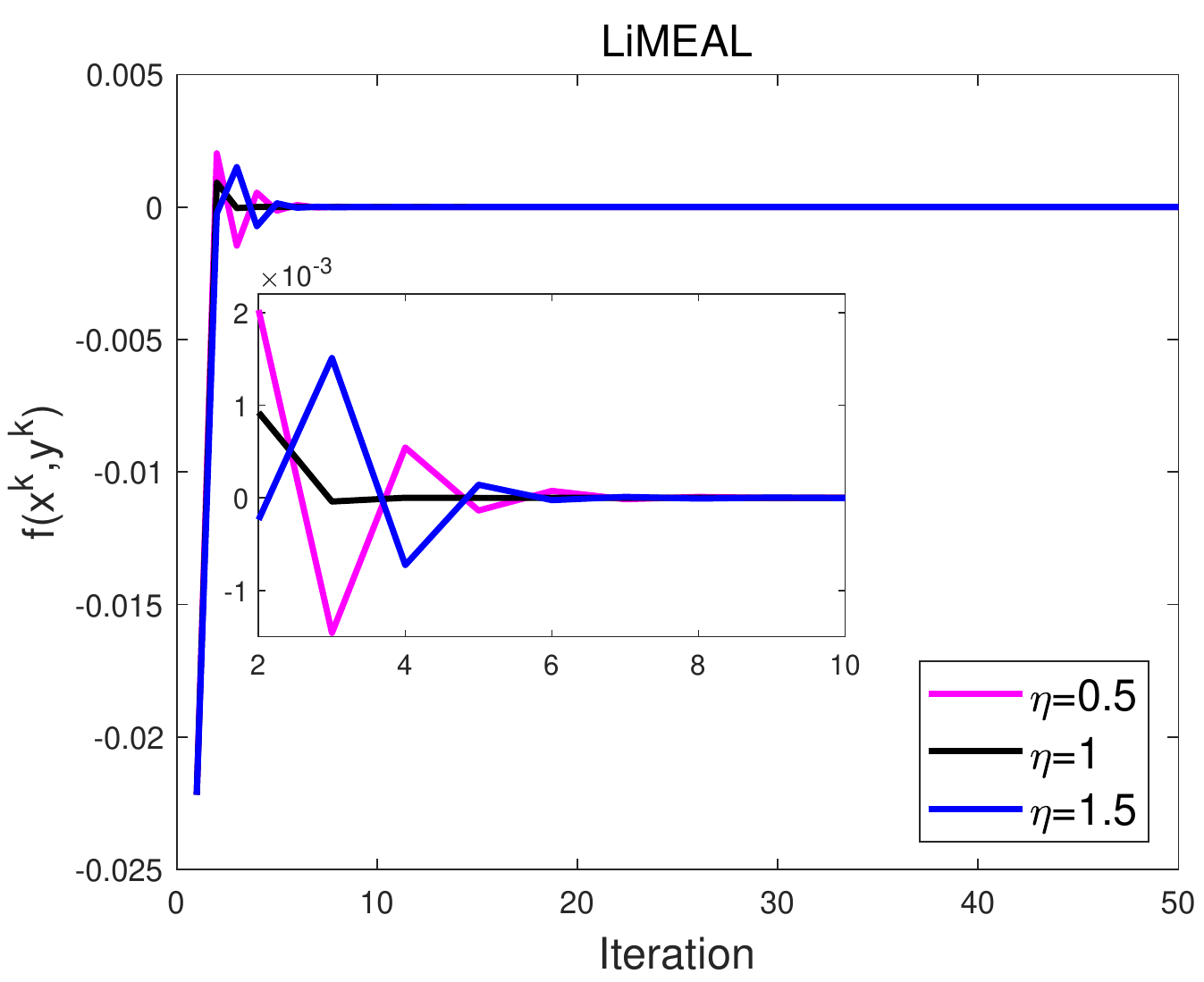}
\centerline{{\small (f) objective sequence of LiMEAL}}
\end{minipage}
\caption{Apply ALM and LiMEAL to problem \eqref{Eq:exp1}. ALM diverges. LiMEAL quickly converges.
}
\label{Fig:Exp1}
\end{figure}

\subsection{Quadratic Programming}
\label{sc:exp2}
Consider the quadratic program with box constraints:
\begin{align}\label{Eq:exp2}
    \min_{x\in \mathbb{R}^n} \ \frac{1}{2} x^TQx + r^Tx \quad s.t. \quad Ax=b, \ \ell_i \leq x_i \leq u_i, \ i=1,\ldots,n,
\end{align}
where $Q\in \mathbb{R}^{n\times n}$, $r\in \mathbb{R}^n$, $A\in \mathbb{R}^{m\times n}$, $b\in \mathbb{R}^m$, and $\ell_i, u_i \in \mathbb{R}$, $i=1,\ldots, n$. Let ${\cal C}:=\{x: \ell_i \leq x_i \leq u_i, i=1,\ldots,n\}$.
Applying LiMEAL yields: initialize $(x^0, z^0,\lambda^0)$, $\gamma>0$, $\eta \in (0,2)$ and $\beta>0$, for $k=0,1,\ldots,$ run
\begin{equation*}
\mathrm{(LiMEAL)} \quad
\left\{
\begin{array}{l}
\tilde{x}^k = (\beta A^TA+\gamma^{-1}{\bf I}_n)^{-1}(\gamma^{-1}z^k+\beta A^Tb-r-Qx^k-A^T\lambda^k),\\
x^{k+1} = \mathrm{Proj}_{\cal C}(\tilde{x}^k),\\
z^{k+1} = z^k - \eta(z^k-x^{k+1}),\\
\lambda^{k+1} = \lambda^k + \beta_k (Ax^{k+1}-b).
\end{array}
\right.
\end{equation*}
Applying \textit{Prox-iALM} from~\cite[Algorithm 2.2]{Zhang-Luo18} yields:
initialize $(x^0, z^0,\lambda^0)$, parameters $\beta, p, \alpha, s, \eta>0$, for $k=0,1,\ldots,$ run
\begin{equation*}
\mathrm{(Prox-iALM)} \quad
\left\{
\begin{array}{l}
\bar{x}^k = (\beta A^TA+p{\bf I}_n)x^k +Qx^k+A^T\lambda^k-p z^k-(\beta A^Tb-r),\\
x^{k+1} = \mathrm{Proj}_{\cal C}(x^k-s\bar{x}^k),\\
z^{k+1} = z^k - \eta(z^k-x^{k+1}),\\
\lambda^{k+1} = \lambda^k + \beta_k (Ax^{k+1}-b).
\end{array}
\right.
\end{equation*}
When $\eta=1$, then Prox-iALM reduces to  \textit{Algorithm 2.1} in \cite{Zhang-Luo18}, which we name \textit{iALM}.

The experimental settings are similar to~\cite[Sec. 6.2]{Zhang-Luo18}: set $m=5, n=20$, generate the entries of $Q$, $A$, $b$, and $\tilde{x}$ by sampling from the uniform distribution, and set $b=A\tilde{x}$. For LiMEAL, we set $\beta = 50, \gamma=\frac{1}{2\|Q\|_2}$ and test three values of $\eta's$: $0.5, 1, 1.5$. For Prox-iALM, we use the parameter settings in \cite[Sec. 6.2]{Zhang-Luo18}: $p = 2\|Q\|_2, \beta = 50, \alpha = \frac{\beta}{4}, s = \frac{1}{2(\|Q\|_2+p+\beta \|A\|_2^2)}$. Moreover, we test two values of $\eta's$: $1$ and $0.5$ for Prox-iALM. Prox-iALM with $\eta=1$ reduces to iALM. The curves of the objective sequence, $\|Ax^k-b\|$, $\|x^{k+1}-z^k\|$ and the norm of gradient of the Moreau envelope are depicted in Fig. \ref{Fig:exp2}.
We observe that LiMEAL converges faster than both iALM and Prox-iALM. By Fig. \ref{Fig:exp2}(d), LiMEAL converges exponentially fast with all three values of $\eta's$. These results verify the results in Proposition \ref{Proposition:globalconv-LiMEAL}(b) since the augmented Lagrangian of problem \eqref{Eq:exp2} is a K{\L} function with an exponent of $1/2$. 

\begin{figure}[!t]
\begin{minipage}[b]{0.49\linewidth}
\centering
\includegraphics*[scale=.43]{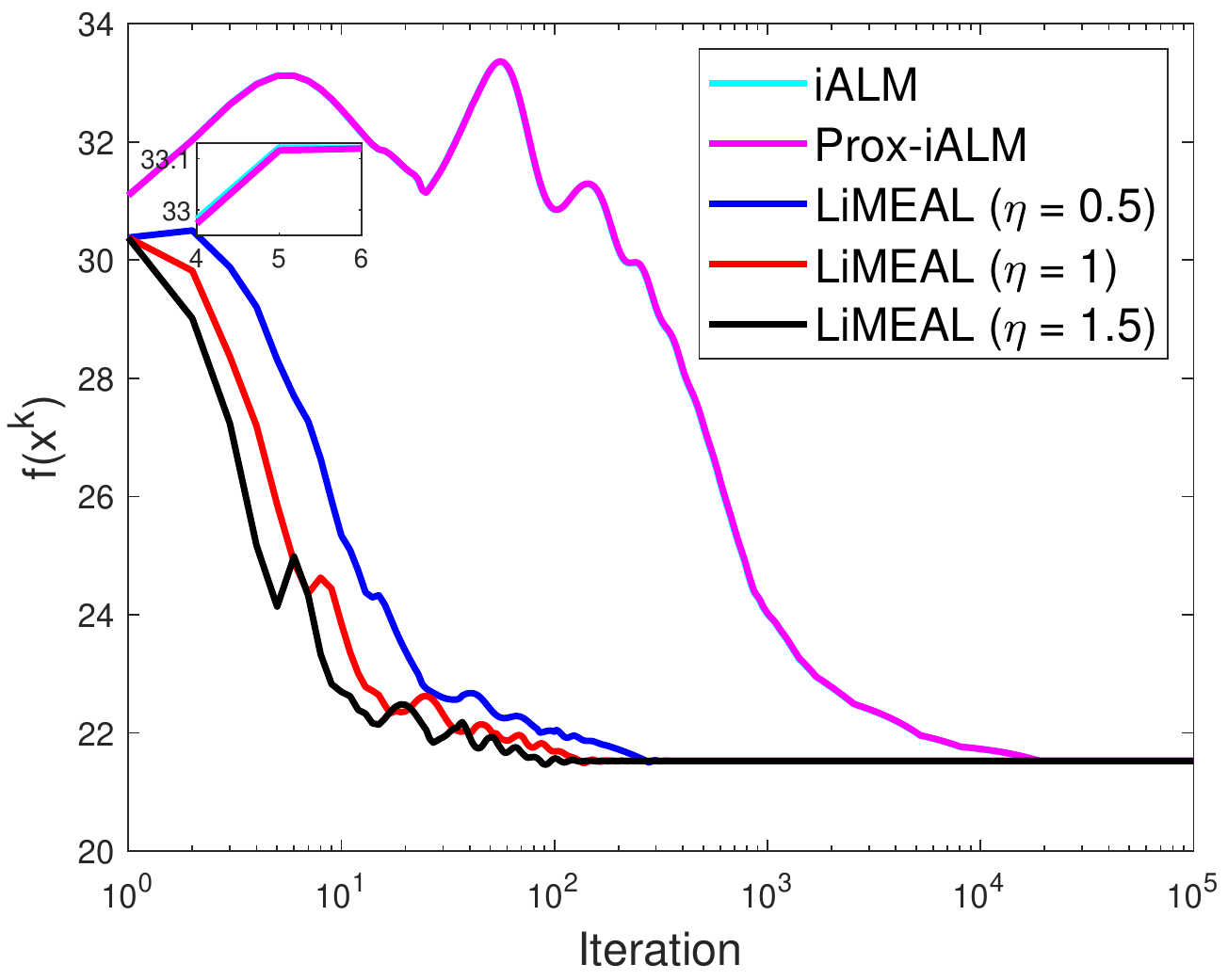}
\centerline{{\small (a) Objective sequence}}
\end{minipage}
\hfill
\begin{minipage}[b]{0.49\linewidth}
\centering
\includegraphics*[scale=.43]{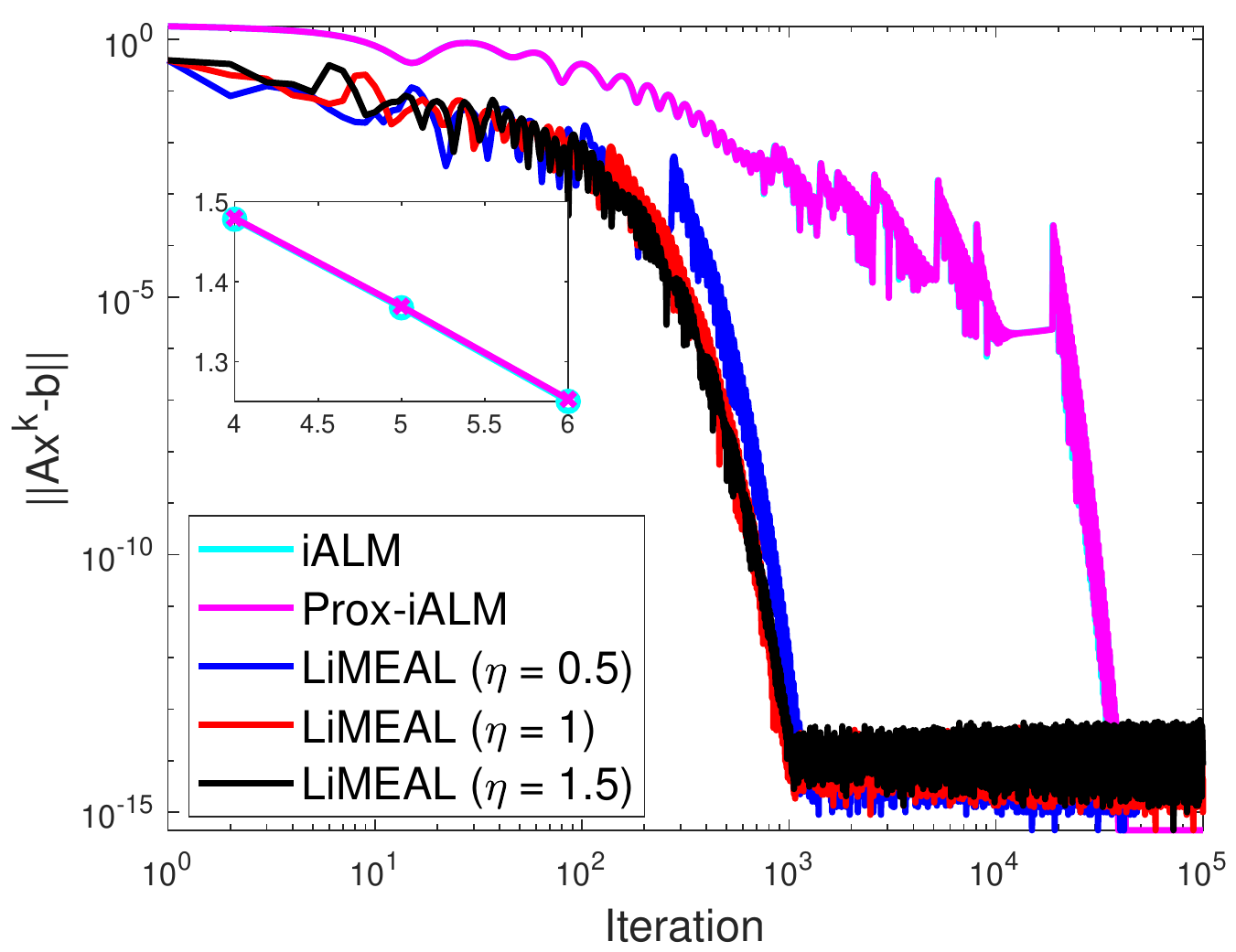}
\centerline{{\small (b) $\|Ax^k-b\|$}}
\end{minipage}
\hfill
\begin{minipage}[b]{0.49\linewidth}
\centering
\includegraphics*[scale=.43]{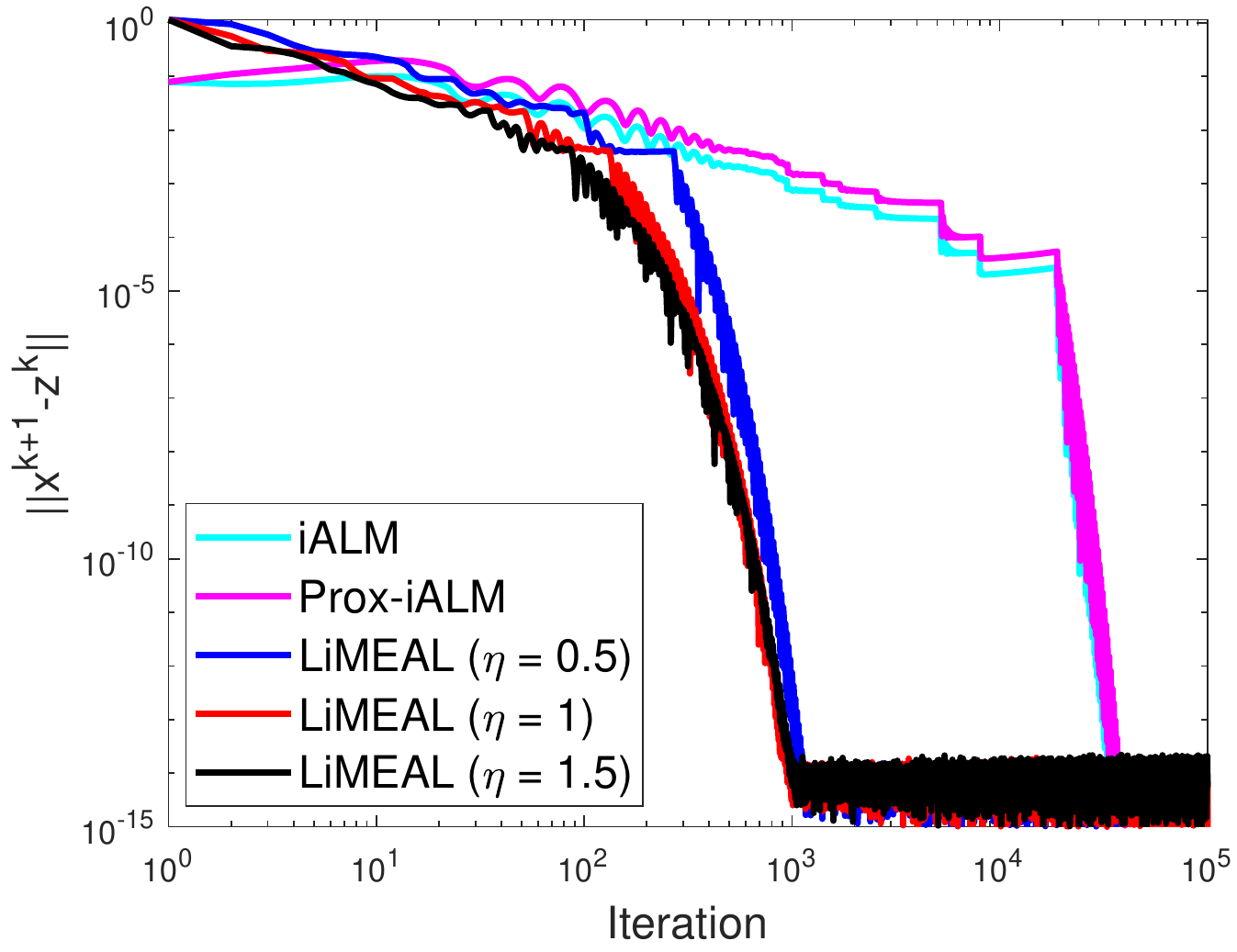}
\centerline{{\small (c) $\|x^{k+1}-z^k\|$}}
\end{minipage}
\hfill
\begin{minipage}[b]{0.49\linewidth}
\centering
\includegraphics*[scale=.43]{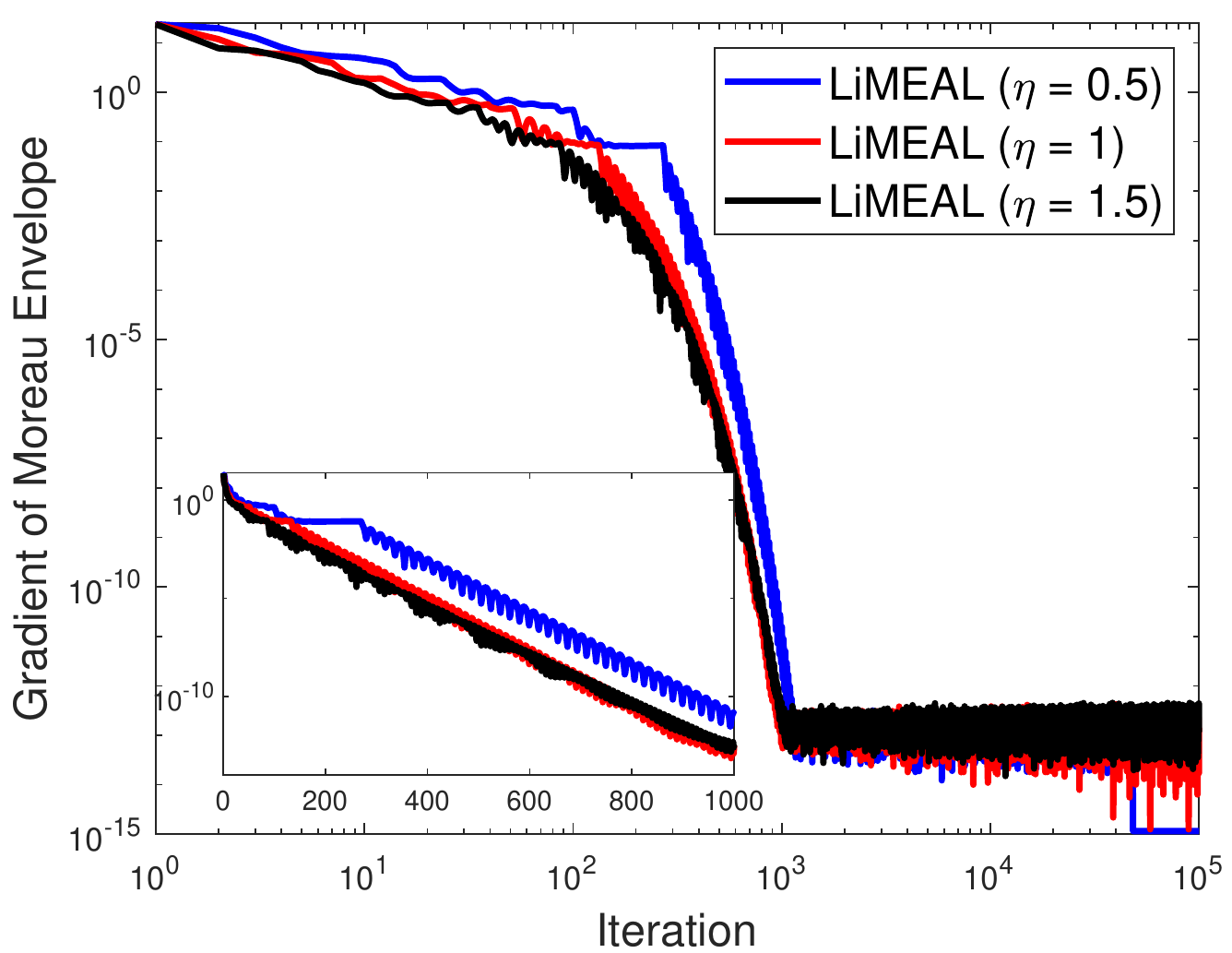}
\centerline{{\small (d) Convergence rates of LiMEAL}}
\end{minipage}
\caption{Performance of LiMEAL and Prox-iALM for the quadratic programming problem \eqref{Eq:exp2}.
}
\label{Fig:exp2}
\end{figure}

\section{Conclusion}
\label{sc:conclusion}

This paper suggests a Moreau envelope augmented Lagrangian (MEAL) method for the linearly constrained weakly convex optimization problem.
By leveraging the \textit{implicit smoothing property} of Moreau envelope, the proposed MEAL generalizes the ALM and proximal ALM to the nonconvex and nonsmooth case.
To yield an $\varepsilon$-accurate first-order stationary point, the iteration complexity of MEAL is $o(\varepsilon^{-2})$ under the \textit{implicit Lipschitz subgradient} assumption and ${\cal O}(\varepsilon^{-2})$ under the \textit{implicit bounded subgradient} assumption.
The global convergence and rate of MEAL are also established under the further Kurdyka-{\L}ojasiewicz inequality.
Moreover, an inexact variant (called \textit{iMEAL}), and a prox-linear variant (called \textit{LiMEAL}) for the composite objective case are suggested and analyzed for different practical settings.
The convergence results established in this paper for MEAL and its variants are generally stronger than the existing ones, but under weaker assumptions.

One future direction of this paper is to get rid of the \textit{implicit Lipschitz subgradient} and \textit{implicit bounded subgradient} assumptions, which in some extent limit the applications of the suggested algorithms, though these two assumptions are respectively weaker than the \textit{Lipschitz differentiable} and \textit{bounded subgradient} assumptions commonly used in the literature.
Another direction is to generalize this work to the constrained problem with nonlinear constraints.
The third direction is to develop more practical variants of the proposed methods as well as establish their convergence results.
One possible application of our study is robustness and convergence of stochastic gradient descent in training parameters of structured deep neural networks such as deep convolutional neural networks \cite{Zhou20}, where linear constraints can be used to impose convolutional structures.
We leave them in our future work.


%
%

\bibliographystyle{spmpsci}      
\bibliography{sample}   


\end{document}